\documentclass[a4paper,11pt,english]{smfart}

\usepackage[english]{babel}
\usepackage[utf8]{inputenc}
\usepackage[T1]{fontenc}
\usepackage{lmodern}
\usepackage{amsmath}
\usepackage{amssymb}
\usepackage{amsthm}
\usepackage{mathtools}
\usepackage{array}
\usepackage{braket}
\usepackage{tikz}
\usepackage{todonotes}
\usepackage{color}
\definecolor{shadecolor}{gray}{0.90}
\RequirePackage{calrsfs}
\DeclareSymbolFont{rsfscript}{OMS}{rsfs}{m}{b}
\DeclareSymbolFontAlphabet{\mathrsfs}{rsfscript}

\usepackage{hyperref}

\usepackage{fancyvrb}

\usepackage[utopia,expert]{mathdesign}

\usepackage{lscape}

\usepackage[final]{showlabels}

\usepackage[leftbars]{changebar}

\usepackage{palatino}

\usetikzlibrary{cd}

\theoremstyle{plain}
\newtheorem{theoreme}{Theorem}[section]
\newtheorem{proposition}[theoreme]{Proposition}
\newtheorem{lemma}[theoreme]{Lemma}
\newtheorem{corollary}[theoreme]{Corollary}
\newtheorem{definition}[theoreme]{Definition}

\theoremstyle{definition}

\theoremstyle{remark}
\newtheorem*{remark}{Remark}

\DeclareMathOperator{\End}{End}
\DeclareMathOperator{\Hom}{Hom}

\DeclareMathOperator{\Rep}{Rep}
\DeclareMathOperator{\id}{id}
\DeclareMathOperator{\tr}{tr}
\DeclareMathOperator{\Tr}{Tr}

\DeclareMathOperator{\Irr}{Irr}
\DeclareMathOperator{\ev}{ev}
\DeclareMathOperator{\coev}{coev}
\DeclareMathOperator{\Gr}{Gr}

\DeclareMathOperator{\rev}{rev}
\DeclareMathOperator{\diag}{diag}

\def\arxivprefix{https://arxiv.org}

\def\boitegrise#1#2{\begin{centerline}{\fcolorbox{black}{shadecolor}{~
    \begin{minipage}[t]{#2}{\vphantom{~}#1\vphantom{$A_{\displaystyle{A_A}}$}}
            \end{minipage}~}}\end{centerline}\medskip}

\title{Crossed $S$-matrices and Fourier matrices for Coxeter groups with automorphism }
\author{{\sc Abel Lacabanne}}
\address{
Institut Montpelliérain Alexander Grothendieck (CNRS: UMR 5149), 
Université de Montpellier,
Case Courrier 051,
Place Eugène Bataillon,
34095 MONTPELLIER Cedex,
FRANCE}
\date{\today}
\email{abel.lacabanne@umontpellier.fr}

\begin{document}

\begin{abstract}
  We study crossed $S$-matrices for braided $G$-crossed categories and reduce their computation to a submatrix of the de-equivariantization. We study the more general case of a category containing the symmetric category $\Rep(A,z)$ with $A$ a finite cyclic group and $z\in A$ such that $z^2=1$. We give two example of such categories, which enable us to recover the Fourier matrix associated with the big family of unipotent characters of the dihedral groups with automorphism as well as the Fourier matrix of the big family of unipotent characters of the Ree group of type ${}^2F_4$.
\end{abstract}

\maketitle 

\pagestyle{myheadings}
\markboth{\sc A. Lacabanne}{\sc Crossed $S$-matrices and Fourier matrices for Coxeter groups with automorphism}

In his classification of unipotent characters of a finite group of Lie type $G(q)$, two tools used by Lusztig are a partition of these characters into families and the notion of a non-abelian Fourier transform. One remarkable property of this classification is that it only depends on the Weyl group of $G$ and not on the number $q$. It then has been suggested by Lusztig \cite{lusztig-unipotent-coxeter} and by Broué-Malle-Michel \cite{broue-malle-michel1,broue-malle-michel2}, that a notion of ``unipotent characters'', a partition into families and a Fourier transform should be associated not only with a Weyl group, but more generally with a complex reflection group.

In the case of a Weyl group, the Fourier matrix can be understood in terms of modular invariants of a modular category, namely the category of modules over the (possibly twisted) Drinfeld double of a finite group. We hope that every Fourier matrix constructed in the more general case of complex reflection groups has a categorical explanation as in \cite{bonnafe-rouquier,lacabanne-slightly,lacabanne-double}. Lusztig \cite[3.8]{lusztig-exotic} explains the Fourier matrix of the big family of unipotent characters for a dihedral group with a category constructed from quantum $\mathfrak{sl}_2$ at an even root of unity. We will detail this construction in Section \ref{sec:WZW} and \ref{sec:degree_0}. In type $H_4$, there exists a Fourier matrix associated with a family of size $74$ for which no categorical interpretation is known.

This article focuses on Fourier matrices associated with families of unipotent characters of groups with automorphism, such as Suzuki and Ree groups \cite{geck-malle}. In this setting, the Fourier matrices are not anymore symmetric and a categorical explanation with a modular category is not possible. Following Deshpande \cite{deshpande}, we associate to a braided $G$-crossed fusion category a crossed $S$-matrix, which is not symmetric and should play the role of the Fourier matrix in this twisted case.

It is possible to understand this crossed $S$-matrix, via equivariantization, in terms of a non-degenerate braided pivotal category $\mathcal{C}$ containing the symmetric category $\Rep(A,z)$ of representations of a cyclic group $A$, together with an involution $z\in A$: the crossed $S$-matrix is obtained as a submatrix of the $S$-matrix of the category $\mathcal{C}$. We construct examples in which this category $\mathcal{C}$ encompasses both the Fourier matrix of the family of unipotent characters for the twisted group and the corresponding family for the non-twisted group.

Such categories were studied by Beliakova-Blanchet-Contreras \cite{beliakova-blanchet-contreras} in order to obtain refinements of the Witten-Reshetikhin-Turaev $3$-manifold invariants. Nevertheless, we will not use their terminology of refinable and spin modular categories since we work with categories which are not necessarily ribbon.

This paper is organized as follows. In Section \ref{sec:braided-g-crossed-twisted}, we start by fixing some notations and then we recall the definition of Deshpande of a crossed $S$-matrix in the case of braided $G$-crossed categories. Since we work in a not necessarily spherical framework, we show in Section \ref{sec:categ-cont-rep-g} that this crossed $S$-matrix, under a non-degeneracy hypothesis, is invertible and satisfies some relation with respect to the twist of the category. In Section \ref{sec:twisted-modular-dihedral}, we construct a category containing $\Rep(\mathbb{Z}/2\mathbb{Z})$ which enables us to recover the Fourier matrix of the big family of unipotent characters for dihedral groups, the eigenvalues of the Frobenius, and similar date for dihedral groups with automorphism. Finally in Section \ref{sec:central-extension}, we show that the category of finite dimensional modules of the Drinfeld double of a central extension $\tilde{G}$ of a finite group $G$ fits into the framework of Section \ref{sec:categ-cont-rep-g}. We relate the modularization of the component of trivial degree with the category of finite dimensional representations of the Drinfeld double of $G$. We conclude this Section with the example of the symmetric group $\mathfrak{S}_4$ and the binary octahedral group, which enables us to recover the Fourier matrix associated with the Ree group of type ${}^2F_4$.

%%% Local Variables:
%%% mode: latex
%%% TeX-master: "article"
%%% End:

\subsection*{Acknowledgements}
  I warmly thank my advisor C. Bonnaf\'e for many fruitful discussions
  and his constant support.

\section{Braided $G$-crossed categories and twisted $S$-matrices}
\label{sec:braided-g-crossed-twisted}

We start by some notation and we recall the definition of a braided $G$-crossed
category, notion due to Turaev \cite{turaev} and
explain how to define a crossed $S$-matrix associated to any $g\in G$,
following \cite{deshpande}. We fix $G$ a finite group with identity $1$.

\subsection{Notations}
\label{sec:notations}

In this paper, we will work over an algebraicailly closed field $\Bbbk$ of characteristic $0$. We recall some classical extra structure for a monoidal category $(\mathcal{C},\otimes,\mathbf{1})$. A left dual of an object $X$ is the datum of $(X^*,\ev_X,\coev_X)$ where $X^*$ is an object of $\mathcal{C}$, $\ev_X\colon X^*\otimes X \rightarrow \mathbf{1}$ and $\coev_X\colon\mathbf{1}\rightarrow X\otimes X^*$ such that the following compositions are identities
  \[
    \begin{tikzcd}[column sep=huge]
      X \arrow{r}{\coev_X \otimes \id_X}&X\otimes X^*\otimes X
      \arrow{r}{\id_X\otimes\ev_X} &X,
  \end{tikzcd}
\]
\[
    \begin{tikzcd}[column sep=huge]
      X^* \arrow{r}{\id_{X^*}\otimes\coev_X}&X^*\otimes X\otimes X^*
      \arrow{r}{\ev_X\otimes\id_{X^*}} &X^*.
  \end{tikzcd}
\]

Similarly, there exists a notion of a right dual. If a left dual exists, it is then unique up to isomorphism. A monoidal category is said to be rigid if it admits both left and right duals. In a rigid category, the choice of a left dual for any object define a contravatiant duality functor $*\colon \mathcal{C}\rightarrow\mathcal{C}$.

We will work with fusion categories as defined in \cite[Definition 4.1.1]{egno}. A fusion category is equipped with a pivotal structure if there is an isomorphism of monoidal functors $a_X\colon X\rightarrow X^{**}$. We also recall the definition of quantum traces in a pivotal fusion category with unit object simple. For $X$ an object and $f$ an endomorphism of $X$, we define the right quantum trace $\tr_{X}^R(f)$ of $f$ as the following endomorphism of $\mathbf{1}$:
\[
\begin{tikzcd}
  \mathbf{1} \ar[r,"\coev_X"] & X\otimes X^* \ar[r,"(a_x\circ f)\otimes \id_{X^*}"] &[2em] X^{**}\otimes X^* \ar[r,"\ev_{X^*}"] & \mathbf{1}.
\end{tikzcd}
\]
Since in a fusion category the unit object is simple, we identify this endomorphism with the unique scalar $\lambda\in\Bbbk$ such that this endomorphism is equal to $\lambda \id_{\mathbf{1}}$. Similarly, we may define the left quantum trace as the composition
\[
\begin{tikzcd}
  \mathbf{1} \ar[r,"\coev_{X^*}"] & X^*\otimes X^{**} \ar[r,"\id_{X^*}\otimes(f\circ a_x^{-1})"] &[2em] X^{*}\otimes X \ar[r,"\ev_{X}"] & \mathbf{1}.
\end{tikzcd}
\]
The pivotal structure is said to be spherical if these two quantum traces coincide. We then define the right and left quantum dimension of an object $X$ by taking the trace of the identity map:
\[
  \dim^R(X) := \tr^R_X(\id_X)\quad\text{and}\quad \dim^L(X) := \tr^L_X(\id_X).
\]
It is known that in a fusion category, the simple objects are of non-zero right and left quantum dimensions. We moreover define the dimension of the category $\mathcal{C}$ as
\[
  \dim(\mathcal{C}) := \sum_{X\in \Irr(\mathcal{C})}\dim^L(X)\dim^R(X),
\]
where $\Irr(\mathcal{C})$ denotes the set of isomorphism classes of simple objects in $\mathcal{C}$.

For a rigid braided tensor category, there exists a natural
isomorphism $u_X \colon X \rightarrow X^{**}$, called the Drinfeld
morphism, defined as the composition
\[
  \begin{tikzcd}
    X\ar{r}{\coev_{X^*}} & X\otimes X^* \otimes
    X^{**}\ar{r}{c_{X,X^*}} & X^*\otimes X\otimes
    X^{**}\ar{r}{\ev_X} & X^{**}.
  \end{tikzcd}
\]
It satisfies for all $X,Y\in\mathcal{C}$,
\[
  u_{X}\otimes u_{Y}= u_{X\otimes Y}\circ c_{Y,X}\circ c_{X,Y}.
\]
To give a pivotal structure $a$ on $\mathcal{C}$ is therefore equivalent to give a \emph{twist} on $\mathcal{C}$, which is a natural isomorphism $\theta_X \colon X \rightarrow X$ satisfying for all $X,Y\in\mathcal{C}$
\[
  \theta_{X\otimes Y}=(\theta_X \otimes \theta_Y) \circ c_{Y,X}\circ c_{X,Y}.
\]
The pivotal structure and the twist are related by $a=u\theta$. We will often endow the braided pivotal category $\mathcal{C}$ with the twist given by the pivotal structure.

\subsection{Braided $G$-crossed fusion categories}
\label{sec:braided-g-crossed}

\begin{definition}
  \label{def:braded-g-crossed-fusion}
  Let $\mathcal{C}$ be a fusion category over $\Bbbk$. We say that $\mathcal{C}$ is
  a braided $G$-crossed fusion category if it is equipped with the following
  structures:
  \begin{enumerate}
  \item a grading $\mathcal{C}=\bigoplus_{g\in G}\mathcal{C}_g$,
  \item an action $g\mapsto T_g$ of $G$ such that
    $T_h(\mathcal{C}_g)\subset\mathcal{C}_{hgh^{-1}}$,
  \item natural isomorphisms for all $g\in G$, $X\in\mathcal{C}_g$ and $Y\in\mathcal{C}$
    \[
      c_{X,Y}\colon X\otimes Y\rightarrow T_{g}(Y)\otimes X,
    \]
   such that axioms similar to the usual hexagons are satisfied
   \cite[(8.106),(8.107)]{egno}, and that the braiding is compatible
   with the action of $G$ \cite[(8.105)]{egno}.
 \end{enumerate}
\end{definition}

A theorem of Kirillov \cite{kirillov} and of Müger \cite{mueger}
affirms that any braided $G$-crossed category can be obtained as the
de-equivariantization of a braided fusion category containing
$\Rep(G)$.

From now on, we fix a braided $G$-crossed fusion category
$\mathcal{C}$ and we denote the grading of an homogeneous object $X$
by $d(X)\in G$. The left dual of $X$ is denoted by $X^*$ and the evaluation and coevaluation maps are respectively denoted by $\ev_X\colon X^*\otimes X\rightarrow\mathbf{1}$ and $\coev_X\colon \mathbf{1}\rightarrow X\otimes X^*$. There exists an analogue of the Drinfeld morphism for
such categories. For any $X\in\mathcal{C}_g$, it is defined as the
composition
\[
\begin{tikzcd} 
  X \ar[r,"\coev_{T_{d(X)^{-1}}(X^*)}"]&[3em] X\otimes T_{d(X)^{-1}}(X^*) \otimes T_{d(X)^{-1}}(X)^{**} \ar[r,"c_{X,T_{d(X)^{-1}}(X^*)}"] &[3em]X^*\otimes X \otimes T_{d(X)^{-1}}(X)^{**} \\
  & \phantom{X\otimes T_{d(X)^{-1}}(X^*) \otimes T_{d(X)^{-1}}(X)^{**}}\ar[r,"\ev_X"] &T_{d(X)^{-1}}(X)^{**},
\end{tikzcd}
\]
and it satisfies
\[
  u_X\otimes u_Y = u_{T_{d(X)d(Y)d(X)^{-1}}(X)\otimes T_{d(X)}(Y)} \circ c_{T_{d(X)}(Y),X} \circ c_{X,Y}.
\]

Since we want to associate to $\mathcal{C}$ and $g\in G$ a crossed
$S$-matrix, we need a pivotal structure on $\mathcal{C}$, which is
moreover compatible with the action of $G$:
\[
a_X\colon X\rightarrow X^{**}\qquad\text{such that}\qquad T_g(a_X)=a_{T_g(X)},
\]
for any $X\in\mathcal{C}_g$. To such a pivotal structure, one can
associate a twist \cite[Section V.2.3]{turaev} by
\[
  a_X = u_{T_{d(X)}(X)}\circ\theta_X.
\]

If the group $G$ is the trivial group, we recover the usual notions of braiding, of pivotal structure and of twist.

\subsection{Crossed $S$-matrix}
\label{sec:crossed-s-matrix}

We recall the definition of \cite{deshpande} of the crossed
$S$-matrix. Here, we fix $g\in G$ and consider $\mathcal{C}_g$ as a
$\mathcal{C}_1$-module category. For any $X$ and $Y$ of respective
degree $1$ and $g$, the double braiding is a morphism
\[
  \begin{tikzcd}
    X\otimes Y\ar[r,"c_{X,Y}"] & Y\otimes X \ar[r,"c_{Y,X}"]
    & T_g(X)\otimes Y.
  \end{tikzcd}
\]
Since we want to compute the trace of this morphism, we will restrict
ourselves to objects $X$ such that $T_g(X)$ is isomorphic to $X$, and
we choose such an isomorphism $\gamma$. By doing so, we can rephrase it by
saying that $X$ and $\gamma$ define an object of the
equivariantization $\mathcal{C}^{\langle g\rangle}$ of $\mathcal{C}$
by the cyclic group generated by $g$. Therefore, we have an
isomorphism
\[
  \begin{tikzcd}
    X\otimes Y\ar[r,"c_{X,Y}"] & Y\otimes X \ar[r,"c_{Y,X}"]
    & T_g(X)\otimes Y\ar[r,"\gamma\otimes\id_Y"]& X\otimes Y,
  \end{tikzcd}
\]
and it is possible to compute the trace using the pivotal structure.

\begin{definition}
  \label{def:crossed-S-mat}
  For $g\in G$, the crossed $S$-matrix $S_{1,g}$ is the matrix with
  rows indexed by the isomorphism classes of objects
  $X\in\mathcal{C}_1$ such that $T_g(X)\simeq X$ and with columns
  indexed by isomorphism classes of objects $Y\in\mathcal{C}_g$ with
  value
  \[
    (S_{1,g})_{X,Y}=\tr^R_{X\otimes Y}((\gamma\otimes\id_Y)\circ c_{Y,X}\circ c_{X,Y}),
  \]
  where $\gamma$ is a chosen isomorphism between $T_g(X)$ and $X$: we
  have chosen a lifting in $\mathcal{C}^{\langle g\rangle}$ for any $X\in\mathcal{C}_1$ such that $T_g(X)\simeq X$.
\end{definition}

\begin{remark}
  \label{rk:change-iso-X}
  If there exists an isomorphism $\gamma$ between $T_g(X)$ and $X$, then, for
  any root of unity $\xi\in\Bbbk$ of order the order of $g$ in $G$,
  the morphism $\xi\gamma$ is also an isomorphism between $T_g(X)$ and
  $X$. Using the isomorphism $\xi\gamma$ instead of $\gamma$
  multiplies the row of $X$ by $\xi$.
\end{remark}

Similarly to \cite{deshpande}, we will work in the category
$\mathcal{C}^{\langle g\rangle}$ in order to state some properties of
this matrix. Since our pivotal structure is not necessarily spherical,
we cannot use directly the results of \cite{deshpande}. Since we are
only interested in the objects of degree $g$ and in the
equivariantization $\mathcal{C}^{\langle g\rangle}$, we may and will
suppose that the group $G$ is cyclic, generated by $g$.

\boitegrise{\textbf{Hypothesis.} \emph{From now on, we suppose that $G$ is
    a cyclic group generated by $g\in G$. We denote by $\hat{G}$ the group of characters of $G$.}}{0.8\textwidth}

The equivariantization $\mathcal{C}^G$ is a braided category
containing $\Rep(G)$, and the braiding between two homogeneous objects
$(X,(u_h)_{h\in G})$ and $(Y,(v_h)_{h\in G})$ is given by
$(u_{g^{d(X)}}\otimes\id_X)\circ c_{X,Y}$. Here $u_h$ is an isomorphism between $T_h(X)$ and $X$ satisfying some relations, see \cite[Section 2.7]{egno} for more details.

\begin{lemma}
  \label{lem:iso-Y-gY}
  Let $\mathcal{C}$ be a braided $G$-crossed category with $G$ cyclic
  generated by $g$ and $Y\in \mathcal{C}_g$. Then there exists an
  isomorphism between $T_g(Y)$ and $Y$. 
\end{lemma}

\begin{proof}
  The braiding gives an isomorphism between $Y^*\otimes T_g(Y)$ and
  $Y\otimes Y^*$. The unit object is therefore a direct summand of
  $Y^*\otimes T_g(Y)$, which is possible only if $T_g(Y)\simeq Y$.
\end{proof}

\begin{proposition}
  \label{prop:submatrix-equiv}
  Let $\mathcal{C}$ be a braided $G$-crossed category, with $G$
  generated by $g$. If we choose a lifting
  $(X,(u_h)_{h\in G})\in\mathcal{C}^G$ for any $X\in\mathcal{C}_1$ such that
  $T_g(X)\simeq X$ and a lifting $(Y,(v_h)_{h\in G})\in\mathcal{C}^G$ for any
  $Y\in\mathcal{C}_g$, the crossed $S$-matrix $S_{1,g}$ is a
  submatrix of the $S$-matrix of the braided fusion category $\mathcal{C}^G$.
\end{proposition}

\begin{remark}
  \label{rk:change-iso-Y}
  If we use another lifting for $Y$, this has no effect on the crossed
  $S$-matrix, since $\gamma_Y$ does not appear in the computation of
  $\tr^R_{X\otimes Y}((u_{g}\otimes\id_Y)\circ c_{Y,X}\circ c_{X,Y})$.
\end{remark}

The category $\mathcal{C}^G$ satisfies $\Rep(G)\subset
\mathcal{C}^G$ and we denote by $X_\alpha$ the simple object
corresponding to the simple representation $\alpha\in\hat{G}$ of
$G$. Tensoring by the invertible objects $X_\alpha$ gives an action of
$\hat{G}$ on the objects of $\mathcal{C}^G$, and this allows us to
define a grading on $\mathcal{C}^G$ by the group $\hat{G}$ (see
\cite[Remark 3.29]{mueger}). For any $h\in G$, denote by $(\mathcal{C}^G)_h$
the full subcategory of objects $X\in\mathcal{C}^G$ such that
\[
  c_{X,X_\alpha}\circ c_{X_\alpha,X}=\alpha(h)\id_{X_\alpha\otimes X},
\]
for any $\alpha\in \hat{G}$.

\begin{proposition}
  \label{prop:grading-equiv}
  This defines a grading on the category $\mathcal{C}^G$ such that any
  lifting $(X,(u_h)_{h\in G})$ of $X\in\mathcal{C}$ of degree $h$ is of degree $h$ in $\mathcal{C}^G$. 
\end{proposition}

\begin{proof}
  Any simple object $X\in\mathcal{C}^G$ is homogeneous. Indeed, since $X_\alpha$ is an invertible object,
  the double braiding $c_{X,X_\alpha}\circ c_{X_\alpha,X}$ is the
  multiplication by a scalar $\lambda_\alpha$. The hexagon axioms for
  the braiding immediately show that $\alpha\mapsto \lambda_\alpha$ is
  a character of $\hat{G}$, and therefore there exists $h\in G$ such
  that $\lambda_\alpha=\alpha(h)$ for any $\alpha\in\hat{G}$.

  Checking that this indeed a graduation is again an application of
  the hexagon axioms for the braiding \cite[Lemma 8.9.1]{egno}.

  Finally, let $X$ be an object in $\mathcal{C}$ of degree $h$ and
  $(X,(u_h)_{h\in G})$ a lifting of $X$ in $\mathcal{C}^{G}$. The double braiding $c_{X,X_\alpha}\circ
  c_{X_\alpha,X}$ between $X_\alpha$ and $(X,(u_h)_{h\in G})$ in $\mathcal{C}^G$ is the morphism
  \[
    \begin{tikzcd}
      \mathbf{1}\otimes X \ar[r,"c_{\mathbf{1},X}"] & X\otimes
      \mathbf{1} \ar[r,"c_{X,\mathbf{1}}"] & T_h(\mathbf{1})\otimes X
      \ar[r,"\alpha(h)\varphi_h\otimes\id_X"] &[2em] \mathbf{1}\otimes X,
    \end{tikzcd}
  \]
  where $\varphi_h\colon T_h(\mathbf{1}) \rightarrow \mathbf{1}$ is
  the isomorphism coming from the tensor structure of the action of
  $G$. The double braiding is then a multiple of the identity. 
\end{proof}

Since we have some choice on the lifting $(X,(u_h)_{h\in G})$ of an object
$X$, it is crucial to understand this choice in terms of the category
$\mathcal{C}^G$. In fact, choosing another lifting is corresponds to the
tensorization in $\mathcal{C}^G$ by an invertible object of the form
$X_{\alpha}$. The simple objects $X\in\mathcal{C}_1$ such that there
exists an isomorphism between $T_g(X)$ and $X$ have exactly $\# G$
non-isomorphic liftings in $\mathcal{C}^G$, and the group $\hat{G}$
consequently acts transitively and freely on these liftings. The other
simple objects in $(\mathcal{C}^G)_1$ have a non simple image in
$\mathcal{C}$ by the forgetful functor $\mathcal{C}^G\rightarrow\mathcal{C}$.

Now, we choose a set of representatives $[I_{1,n}]$ of the set of
isomorphism classes of simple objects in $(\mathcal{C}^G)_1$ with trivial
stabilizer under tensorization by the $X_\alpha$'s, under the action
of $\hat{G}$. Similarly, we choose a set of representatives 
$[I_g]$ of the set of isomorphism classes of simple objects in
$(\mathcal{C}^G)_g$ under the action of $\hat{G}$. With these choices, we make explicit Proposition \ref{prop:submatrix-equiv}.

\begin{proposition}
  \label{prop:submatrix-equiv-explicit}
  Let $\mathcal{C}$ be a braided $G$-crossed category, where $G$ is a
  finite cyclic group generated by $g$. The crossed $S$-matrix
  $S_{1,g}$ is the submatrix $(S_{X,Y})_{X\in [I_{1,n}],Y\in [I_g]}$
  of the $S$-matrix of $\mathcal{C}^G$.
\end{proposition}

Similarly, the category $\mathcal{C}^G$ is equipped with a twist,
which does not depend on the chosen lifting of $X\in\mathcal{C}_1$ and
depends of the chosen lifting of $Y\in\mathcal{C}_g$, choosing another
lifting multiplies the twist of $Y$ by a root of unity of order $\# G$.

%%%Local Variables:
%%% mode: latex
%%% TeX-master: "article"
%%% End:

\section{Categories containing $\Rep(A,z)$ with $A$ cyclic}
\label{sec:categ-cont-rep-g}

As in \cite{deshpande}, we have reduced the computation of the crossed
$S$-matrix to a computation in a pivotal braided fusion category
containing the category of representations of a finite cyclic group. We will
state some properties of this matrix, and work in the slightly
extended setting where the braided fusion category contains the
symmetric category $\Rep(A,z)$, for $A$ finite, cyclic and $z\in A$ is of square $1$.

\boitegrise{\textbf{Hypothesis.} \emph{We fix $\mathcal{C}$ a braided pivotal fusion category over $\Bbbk$ such that $\Rep(A,z)\subset \mathcal{C}$ as a braided pivotal subcategory.}}{0.8\textwidth}

Recall that we endow the category $\mathcal{C}$ with a grading by the group $A$, and we denote the degree of a homogeneous object $X$ by $d(X)$. We start with a useful immediate lemma.

\begin{lemma}
  \label{lem:tensorisation-rep-G}
  Let $\mathcal{C}$ be a braided pivotal fusion category containing $\Rep(A,z)$ with $A$ cyclic and $z\in A$ of square $1$. Denote by $X_\alpha$ the simple object in $\mathcal{C}$ corresponding to the representation $\alpha\in\hat{A}$. 
  \begin{enumerate}
  \item For all simple objects $X$ and $Y$ in $\mathcal{C}$ and $\alpha\in\hat{A}$, we have $S_{X_\alpha\otimes X,Y}=\alpha(z)\alpha(d(Y))S_{X,Y}$.
  \item For any twist $\tilde{\theta}$ on $\mathcal{C}$, any simple object $X$ and any $\alpha\in\hat{A}$, we have $\tilde{\theta}_{X_\alpha\otimes X}=\alpha(d(X))\tilde{\theta}_{X_\alpha}\tilde{\theta}_X$.
  \end{enumerate}
\end{lemma}

\subsection{Non-degeneracy and square of the crossed $S$-matrix}
\label{sec:non-deg-square}

Similarly to the case of a pre-modular category, the crossed $S$-matrix will have interesting proprieties if the category $\mathcal{C}$ is non-degenerate.

\begin{remark}
  \label{rk:non-degeneracy-equiv}
  If we start from a braided $G$-crossed category, its equivariantization is non-degenerate if and only if the grading is faithful and the degree $1$ component is non-degenerate \cite[Remark 8.24.4]{egno}.
\end{remark}

\boitegrise{\textbf{Hypothesis.} \emph{We now suppose that the category $\mathcal{C}$ is non-degenerate.}}{0.8\textwidth}

\begin{lemma}
  \label{lem:faithful-grading}
  Let $\mathcal{C}$ be a non-degenerate braided pivotal fusion category containing $\Rep(A,z)$ with $A$ cyclic. Then the grading by $\hat{A}$ is faithful and the symmetric center of $\mathcal{C}_1$ is $\Rep(A,z)$.
\end{lemma}

\begin{proof}
  We denote by $\mathcal{Z}_{\mathcal{C}}(\mathcal{C}_1)$ the centralizer of $\mathcal{C}_1$ in $\mathcal{C}$ and by $\mathcal{Z}_{\mathrm{sym}}(\mathcal{C}_1)$ the symmetric center of $\mathcal{C}_1$. By \cite[Theorem 8.21.1]{egno}, we have
\[
  \dim(\mathcal{C}_1)\dim(\mathcal{Z}_{\mathcal{C}}(\mathcal{C}_1)) = \dim(\mathcal{C})\dim(\mathcal{C}_1\cap\mathcal{Z}_{\mathrm{sym}}(\mathcal{C})).
\]

Since $\mathcal{C}$ is non-degenerate, we have $\dim(\mathcal{C}_1\cap\mathcal{Z}_{\mathrm{sym}}(\mathcal{C}))=1$. But, by definition of $\mathcal{C}_1$, $\Rep(G,z)$ is in the centralizer $\mathcal{Z}_{\mathcal{C}}(\mathcal{C}_1)$. All the categorical dimensions being totally positive numbers, we choose an embedding $\Bbbk_{\mathrm{alg}}\rightarrow\mathbb{C}$, and we have
\[
  \# A \leq \dim(\mathcal{Z}_{\mathcal{C}}(\mathcal{C}_1)) \leq \frac{\dim(\mathcal{C})}{\dim(\mathcal{C}_1)}.
\]

Similar to \cite{egno}, one can prove that for any $g\in G$, $\dim(\mathcal{C}_g)=\dim(\mathcal{C}_1)$ and hence
\[
   \# A \leq \dim(\mathcal{Z}_{\mathcal{C}}(\mathcal{C}_1)) \leq \#A.
\]
Therefore the grading by $A$ is faithful and $\mathcal{Z}_{\mathcal{C}}(\mathcal{C}_1)=\Rep(A,z)$.
\end{proof}

The pivotal structure is not spherical, and there exists an involution $\bar{\phantom{a}}$ on the set $\Irr(\mathcal{C})$: for any simple object $X$, there exists, up to isomorphism, a unique simple object $\bar{X}$ such that for every simple object $Y$,
\[
  \frac{S_{X,Y^*}}{\dim^R(X)}=\frac{S_{\bar{X},Y}}{\dim^R(\bar{X})}.
\]

See \cite{lacabanne-slightly} for more details. It is proven that $\bar{X}\simeq X^*\otimes\bar{\mathbf{1}}$ and that $S_{\bar{X},Y}=S_{X,\bar{Y}}$. We denote by $s_X^R$ the character of the Grothendieck group of $\mathcal{C}$ satisfying $\dim^R(X)s_X^R(Y)= S_{X,Y}$ for all simple objects $Y$. The definition of $\bar{\phantom{a}}$ translates into the equality $s_{\bar{X}}^R(Y)=s_{X}^R(Y^*)$ for any simple object $Y$.

\begin{lemma}
  \label{lem:1bar-trivial-degree}
  The element $\bar{\mathbf{1}}$ is in $\mathcal{C}_1$. Therefore the involution $\bar{\phantom{a}}$ is a bijection between $\mathcal{C}_a$ and $\mathcal{C}_{a^{-1}}$ for any $a\in A$.
\end{lemma}

\begin{proof}
  Since $X_\alpha$ is of dimension $\alpha(z)=\pm 1$, the dimension of $X_\alpha$ and its dual are equal and we have:
\[
  S_{X_\alpha,\bar{\mathbf{1}}}=\dim^R(\bar{\mathbf{1}})\frac{S_{X_\alpha^*,\mathbf{1}}}{\dim^R(\mathbf{1})}=\dim^R(X_\alpha)\dim^R(\bar{\mathbf{1}}).
\]
Therefore $c_{\bar{\mathbf{1}},X_\alpha}\circ c_{X_\alpha,\bar{\mathbf{1}}}=\id_{X_\alpha,\bar{\mathbf{1}}}$ and $\bar{\mathbf{1}}$ is of degree $1$.
\end{proof}

\begin{lemma}
  \label{lem:fixed-points-trivial}
  Let $\alpha\in\hat{A}$ and $X$ be an object of degree $a$. If $X\otimes X_\alpha\simeq X$ then $\alpha(a)=1$.
\end{lemma}

\begin{proof}
  This follows from Lemma \ref{lem:tensorisation-rep-G} using the twist $\theta$ associated with the pivotal structure and the braiding.
\end{proof}

Now, we fix $a\in A$ such that $a$ generates the cyclic group $A$. We denote by $I_1$ the set of isomorphism classes of simple objects of degree $1$ and by $I_{1,n}$ the subset of elements of $I_1$ on which $\hat{A}$ acts with trivial stabilizers. We choose a set of representatives $[I_{1,n}]$ of $I_{1,n}$ under the action of $\hat{A}$.

Similarly, denote by $I_{a^{\pm 1}}$ the set of isomorphism classes of simple objects of degree $a^{\pm 1}$. By Lemma \ref{lem:faithful-grading}, the group $\hat{A}$ acts by tensorization on $I_{a^{\pm 1}}$ with trivial stabilizers. We choose a set of representatives $[I_{a}]$ of $I_{a}$ under the action of $\hat{A}$. We set $[I_{a^{-1}}] = \{\bar{X}\ \vert\ X\in[I_{a}]\}$. This is a set of representatives of $I_{a^{-1}}$ under the action of $\hat{A}$.

\begin{remark}
  \label{rk:compatible-choices}
  \begin{enumerate}
  \item The choice of $[I_{1,n}]$ is not necessarily compatible with the involution $\bar{\phantom{a}}$, but for all $X\in[I_{1,n}]$ there exist a unique $Y\in [I_{1,n}]$ and a unique $\alpha\in\hat{A}$ such that $\bar{X}\simeq X_\alpha\otimes Y$.
  \item If $\# A=2$, then $a=a^{-1}$ and we may and will choose $[I_{a}]$ such that the two sets $[I_{a}]$ and $[I_{a^{-1}}]$ are equal. Indeed, for any $X\in I_a$ and $\alpha\in\hat{A}$, if $\bar{X}\simeq X_\alpha\otimes X$ then $\alpha=1$. This follows from the value of the twist: $\theta_{\bar{X}}=\theta_{X}$ and $\theta_{X_\alpha\otimes X}=\alpha(a)\theta_{X}$.  
  \end{enumerate}
\end{remark}

If $X\in I_1\setminus I_{1,n}$ then for any $Y\in I_{a^{\pm 1}}$ we have $S_{X,Y} = 0$. Indeed, let $\alpha\in\hat{A}\{1\}$ such that $X_\alpha\otimes X\simeq X$. Since $\dim^R(X_\alpha)=\alpha(z)$ and $\dim^R(X)\neq 0$, we must have $\alpha(z)=1$ and
\[
  S_{X,Y} = S_{X_\alpha\otimes X,Y} = \alpha(z)\alpha(g^{\pm 1}) S_{X,Y} = \alpha(a^{\pm 1}) S_{X,Y}.
\]
As $\alpha\neq 1$ and $a$ generates $A$, we have $\alpha(a^{\pm 1})\neq 1$ and $S_{X,Y}=0$.

We then consider the following submatrices of the $S$-matrix of $\mathcal{C}$:
\[
  S_{1,a^{\pm 1}}=(S_{X,Y})_{X\in[I_{1,n}],Y\in [I_{a^{\pm 1}}]},\quad S_{a^{\pm 1},1}={}^t\!S_{1,a^{\pm 1}} \quad \text{and} \quad S_{a^{\pm 1},g} = (S_{X,Y})_{X\in [I_{a^{\pm 1}}],Y\in [I_g]}.
\]
We also define two square monomial matrices $P=(P_{X,Y})_{X,Y \in [I_{1,n}]}$ and $Q=(Q_{W,Z})_{W\in [I_{a^{-1}}], y\in [I_a]}$ by:
\[
  P_{X,Y} = \sum_{\alpha\in\hat{A}} \alpha(z)\alpha(g^{-1})\delta_{\bar{Y}\simeq X_{\alpha}\otimes X}\quad\text{and}\quad Q_{W,Z} = \delta_{\bar{W}\simeq Z},
\]
for any $X,Y\in [I_{1,n}]$, $W\in [I_{a^{-1}}]$ and $Z\in [I_a]$.

\begin{theoreme}
  \label{thm:S1gSg1}
  Let $\mathcal{C}$ be a non-degenerate braided pivotal fusion category containing $\Rep(A,z)$. Then
\[
  S_{1,a}S_{a,1}=\frac{\dim(\mathcal{C}_1)}{\# A}\dim^R(\bar{\mathbf{1}})P \quad\text{and}\quad S_{a^{-1},1}S_{1,a}=S_{a^{-1},a}S_{a,a}=\frac{\dim(\mathcal{C}_1)}{\# A}\dim^R(\bar{\mathbf{1}})Q.
\]

In particular, the matrices $S_{1,a^{\pm1}}$ and $S_{a^{\pm 1},1}$ are square matrices and $\#I_{1,n} = \#I_a$.
\end{theoreme}

\begin{proof}
  Let us start by the product $S_{a^{-1},1}S_{1,a}$. Let $X\in [I_a]$ and $Y\in [I_{a^{-1}}]$. As noticed before, for any $Z\in I_1\setminus I_{1,n}$, one have $S_{Y,Z}=0=S_{Z,X}$. Consequently
\[
  \sum_{Z\in I_1} S_{Y,Z}S_{Z,X} = \sum_{Z \in I_{1,n}}S_{Y,Z}S_{Z,X} = \sum_{Z \in [I_{1,n}]}\sum_{\alpha\in\hat{A}}S_{Y,X_\alpha\otimes Z}S_{X_\alpha\otimes Z,X}.
\]
But $S_{Y,X_\alpha\otimes Z} = \alpha(z)\alpha(a^{-1})S_{Y,Z}$ and $S_{X_\alpha\otimes Z,X} = \alpha(z)\alpha(a)S_{Z,Y}$. Therefore
\[
  (S_{a^{-1},1}S_{1,a})_{Y,X} = \frac{1}{\# A}\sum_{Z\in I_1}S_{Y,Z}S_{Z,X}.
\]

If $X$ and $\bar{Y}$ are not isomorphic, then 
\[
  \sum_{Z\in I_1}S_{Y,Z}S_{Z,X} = \dim^R(X)\dim^R(Y)\sum_{Z\in I_1}s_X^R(Z)s_{\bar{Y}}^R(Y^*)=0,
\] 
since the characters $s_X^R$ and $s_{\bar{Y}}^R$ of $\Gr(\mathcal{C}_1)$ are different \cite[Lemma 8.20.9]{egno}.

If $X$ and $\bar{Y}$ are isomorphic, then
\[
  \sum_{Z\in I_1}S_{Y,Z}S_{Z,X} = \sum_{\substack{Z\in I_1\\W\in I_1}}\dim^R(Z)N_{X,Y}^WS_{Z,W}=\sum_{W\in I_1}N_{X,Y}^W\dim^R(W)\sum_{Z\in I_1}s^R_{\mathbf{1}}(Z)s^R_{\bar{W}}(Z^*).
\] 
The sum over $Z$ is non-zero if and only if $\bar{W}\simeq X_\alpha$ for some $\alpha\in\hat{A}$ and
\[
  \sum_{Z\in I_1}S_{Y,Z}S_{Z,X}=\dim(\mathcal{C}_1)\sum_{\alpha\in\hat{A}}\dim^R(\overline{X_\alpha})N_{X,Y}^{X_\alpha}.
\]
Finally, $N_{X,Y}^{\overline{X_\alpha}}=N_{X,Y\otimes \overline{X_\alpha}^*}^{\mathbf{1}}$ and $Y\otimes \overline{X_\alpha}^*\simeq X^*\otimes X_\alpha$ so that $N_{X,Y}^{\overline{X_\alpha}}$ is zero if $\alpha\neq 1$ and is $1$ if $\alpha=1$. This concludes the computation of $S_{a^{-1},1}S_{1,a}$.

Now, we look at the product $S_{1,a}S_{a,1}$. For $X$ and $Y$ two elements of $[I_1]$, we have
\[
  \sum_{Z\in I_a}S_{X,Z}S_{Z,Y}=\# A \sum_{Z\in [I_g]}S_{X,Z}S_{Z,Y}.
\]

Suppose first that for all $\alpha\in\hat{A}$, the objects $\bar{Y}$ and $X_\alpha\otimes X$ are not isomorphic. Then the characters $s_{X}^R$ and $s_{\bar{Y}}^R$ are different and $\sum_{Z\in I_1}s_X^R(Z)s_{\bar{Y}}^R(Z^*) = 0$. By computing in $\Gr(\mathcal{C})$ the product
\[
  \left(\sum_{Z\in I_1}s_X^R(Z^*)[Z]\right)\left(\sum_{Z\in I_a}s_{\bar{Y}}^R(Z^*)[Z]\right),
\]
we obtain that
\[
  \left(\sum_{Z\in I_1}s_{\bar{Y}}(Z)s_X(Z^*)\right)\left(\sum_{Z\in I_a}s_{\bar{Y}}^R(Z^*)[Z]\right)
  =
  \left(\sum_{Z\in I_a}s_{X}(Z)s_{\bar{Y}}(Z^*)\right)\left(\sum_{Z\in I_a}s_{X}(Z^*)[Z]\right).
\]
Since the first term is zero, we obtain that $\sum_{Z\in I_a}s_{X}(Z)s_{\bar{Y}}(Z^*)=0$ and therefore 
\[
  \sum_{Z\in I_a}S_{X,Z}S_{Z,Y}=0.
\]

Suppose now that $\bar{Y}$ and $X_\alpha\otimes X$ are isomorphic for a certain $\alpha\in\hat{A}$. Similarly to the computation of $S_{a^{-1},1}S_{1,a}$, we have:
\[
  \sum_{Z\in I_a}S_{X,Z}S_{Z,Y} = \sum_{\substack{Z\in I_a\\W\in I_1}}N_{X,Y}^W\dim^R(Z)S_{Z,W} = \sum_{W\in I_1}N_{X,Y}^W\dim(W)\sum_{Z\in I_a}s_{\mathbf{1}}^R(Z)s_{\bar{W}}^R(Z^*),
\]
and the sum over $I_a$ is non-zero if and only if $\bar{W}\simeq X_\beta$ for some $\beta\in\hat{A}$. Therefore
\[
  \sum_{Z\in I_a}S_{X,Z}S_{Z,Y} = \dim(\mathcal{C}_a)\sum_{\beta\in\hat{A}}N_{X,Y}^{\overline(X_\beta)}\dim^R(\overline{X_\beta}).
\]
Finally, $N_{X,Y}^{\overline(X_\beta)}=1$ if and only if $\alpha=\beta$ in $\hat{A}$ and is zero otherwise. Hence
\[
  \sum_{Z\in I_a}S_{X,Z}S_{Z,Y} = \dim(\mathcal{C}_1)\dim^R(\bar{\mathbf{1}})\alpha(a^{-1})\alpha(z),
\]
as stated.

The computations for the last product are similar.
\end{proof}

\begin{remark}
  When $\#A=2$, if we have made a choice of $[I_a]$ which is not stable with respect to the involution $\bar{\phantom{a}}$, we have to add signs to the permutation matrix appearing in the product $S_{a,a}S_{a,a}$.
\end{remark}

\begin{corollary}
  \label{cor:twisted-verlinde}
  Under the same hypothesis, for any $X,Y,Z\in [I_{1,n}]$, we have
\[
  \sum_{\alpha\in\hat{A}}\alpha(a)\alpha(z)N_{X,Y}^{X_\alpha\otimes Z} = \frac{\# A}{\dim(\mathcal{C}_1)\dim^R(\bar{\mathbf{1}})}\sum_{W\in[I_a]}\frac{S_{X,W}S_{Y,W}S_{\bar{Z},W}}{\dim^R(W)}.
\]
\end{corollary}

\begin{proof}
  Let $\alpha\in\hat{A}$ such that $X_\alpha\otimes \bar{Z}\in[I_{1,n}]$. We then have:
  \begin{align*}
    \sum_{W\in[I_a]}\frac{S_{X,W}S_{Y,W}S_{\bar{Z},W}}{\dim^R(W)} 
    &= \sum_{\substack{W\in[I_a]\\U\in I_1}}N_{X,Y}^US_{U,W}S_{W,\bar{Z}}\\
    &= \sum_{U\in [I_{1,n}]}\alpha(a^{-1})\alpha(z)\left(\sum_{\beta\in\hat{A}}\beta(a)N_{X,Y}^{X\alpha\otimes U}\right)(S_{1,a}S_{a,1})_{U,X_\alpha\otimes \bar{Z}}.
  \end{align*}
  
  But the Theorem \ref{thm:S1gSg1} shows that $(S_{1,a}S_{a,1})_{U,X_\alpha\otimes \bar{Z}}=\frac{\dim(\mathcal{C}_1)}{\# A}\dim^R(\bar{\mathbf{1}})\alpha(a)\alpha(z)\delta_{U,Z}$, which leads to the expected formula.
\end{proof}

\begin{remark}
  \label{rk:fusion-ring}
 The ring $\Bbbk\otimes_{\mathbb{Z}}\Gr(\mathcal{C})/([X_\alpha]-\alpha(a)[\mathbf{1}])$ has a basis labelled by the elements of $[I_{1,n}]$. The numbers $\sum_{\alpha\in\hat{A}}\alpha(a)\alpha(z)N_{X,Y}^{X_\alpha\otimes Z}$ are precisely the structure constants of this ring with respect to this basis, see \cite{deshpande}. These constants lie in a cyclotomic ring, which is the integers if $A$ is of order $2$.
\end{remark}

\subsection{Now enters the twist}
\label{sec:twist-graduation}

The twist $\theta$ plays an important role in the representation of $SL_2(\mathbb{Z})$ afforded by a modular category. In our graded setting, the twist of a simple object $X\in [I_{1,n}]$ does not depend on the choice of the representative, whereas the twist of a simple object $Y\in [I_a]$ depends on the choice of the representative, by a root of unity of order $\# A$. We define the following diagonal matrices:
\[
  T_1=\diag(\theta_X^{-1})_{X\in [I_{1,n}]}\quad\text{and}\quad T_{a^{\pm 1}} = \diag(\theta_Y^{-1})_{Y\in [I_{a^{\pm 1}}]}.
\]
Since for any simple object $X$ we have $\theta_{\bar{X}}=\theta_X$, we have $T_{g^{-1}}Q=QT_g$, where $Q$ is the matrix defined before the Theorem \ref{thm:S1gSg1}. We aim to show some relations satisfied by the matrices $S_{a^{-1},1},S_{1,a},T_1,T_a$ and $T_{a^{-1}}$.

The Gauss sums of a fusion category $\mathcal{C}$ with twist $\tilde{\theta}$ are
\[
  \tau^{\pm}(\mathcal{C},\tilde{\theta}):=\sum_{X\in\Irr(\mathcal{C})}\theta_X^{\pm 1}\lvert X \rvert^2.
\]
If $\mathcal{C}$ is a braided pivotal fusion category, we denote by $\tau^{\pm}(\mathcal{C})$ the Gauss sums with twist $\theta$ associated with the pivotal structure.

\begin{lemma}
  \label{lem:tau+tau-}
  Let $\mathcal{C}$ be a non-degenerate braided pivotal fusion category containing $\Rep(A,z)$ with $A$ cyclic. Then
\[
  \frac{1}{\#A}\tau^{+}(\mathcal{C}_1)\frac{1}{\#A}\tau^-(\mathcal{C}_1) = \frac{1}{\#A}\dim(\mathcal{C}_1).
\]
\end{lemma}

\begin{proof}
  If $z=1$, then this follows from the corresponding result for the modularization $\mathcal{D}$ of $\mathcal{C}_1$ since $\frac{1}{\#A}\tau^{\pm}(\mathcal{C}_1)=\tau^{\pm}(\mathcal{D})$ and $\frac{1}{\#A}\dim(\mathcal{C}_1)=\dim(\mathcal{D})$ (see \cite[Proposition 4.4]{bruguieres} for these equalities).

  If $z\neq 1$ then necessarily $\#A$ is even, equal to $2n$. The symmetric center of $\mathcal{C}_1$ is $\Rep(A,z)$ and we consider the de-equivariantization $\mathcal{D}$ of $\mathcal{C}_1$ with respect to the action of $A/\langle z\rangle$. Since $A$ is commutative, we have $\frac{1}{n}\tau^{\pm}(\mathcal{C}_1)=\tau^{\pm}(\mathcal{D})$ and $\frac{1}{n}\dim(\mathcal{C}_1)=\dim(\mathcal{D})$. The category $\mathcal{D}$ is slightly degenerate and the result follows from \cite{lacabanne-slightly}.
\end{proof}

We now give a formula which relates the matrices $S_{1,a}, S_{a^{-1},1},S_{a^{-1},a}, T_1,T_a$ and $T_{a^{-1}}$.

\begin{proposition}
  \label{prop:S-and-T}
  Let $\mathcal{C}$ be a non-degenerate braided pivotal fusion category containing $\Rep(A,z)$ with $A$ cyclic. Then
  \begin{equation}
    S_{a^{-1},1}T_1S_{1,a}=\frac{\tau^{-}(\mathcal{C}_1)}{\# A} T_{a^{-1}}^{-1}S_{a^{-1},a}T_a^{-1}.\label{eq:S-and-T}
  \end{equation}
\end{proposition}

\begin{proof}
  Let $X\in[I_a]$ and $Y\in[I_{a^{-1}}]$. First we have
\[
  (S_{a^{-1},1}T_1S_{1,a})_{Y,X} = \sum_{Z\in [I_{1,n}]}S_{Y,Z}\theta_{Z}^{-1}S_{Z,X} = \frac{1}{\# A}\sum_{Z\in I_1}S_{Y,Z}\theta_{Z}^{-1}S_{Z,X},
\]
since $S_{Y,Z}=0=S_{X,Z}$ for all $Z\in I_1\setminus I_{1,n}$ and $S_{Y,X_\alpha\otimes Z}\theta_{X_\alpha\otimes Z}^{-1}S_{X_\alpha\otimes Z,X} = S_{Y,Z}\theta_{Z}^{-1}S_{Z,X}$ for all $Z\in I_1$ and $\alpha\in\hat{A}$. Then, using \cite{lacabanne-slightly}, we have
\begin{align*}
  \sum_{Z\in I_1}S_{Y,Z}\theta_{Z}^{-1}S_{Z,X} 
  &= \sum_{W\in I_1}N_{Y,X}^W\sum_{Z\in I_1}\theta_{Z}^{-1}\dim^R(Z)S_{Z,W}\\
  &= \tau^-(\mathcal{C}_1)\sum_{W\in I_1}N_{Y,X}^W\dim^R(W)\theta_W\\
  &= \tau^-(\mathcal{C}_1)\theta_X\theta_YS_{X,Y},
\end{align*}
which concludes the proof.
\end{proof}

\subsubsection{When $A=\mathbb{Z}/2\mathbb{Z}$}
\label{sec:n=2}

When the order of $A$ is $2$, the fact that $a=a^{-1}$ enables us to show some extra relation between the matrices defines above.

\begin{theoreme}
  \label{thm:S-and-T-order-2}
  Let $\mathcal{C}$ be a non-degenerate braided pivotal fusion category containing $\Rep(A,z)$, with $A$ cyclic of order $2$. We choose $\sqrt{\dim(\mathcal{C}_1)}$ a square root of $\dim(\mathcal{C}_1)$ in $\Bbbk$, as well as $\sqrt{\dim^R(\bar{\mathbf{1}})}$ a square root of $\dim^R(\bar{\mathbf{1}})$. We define:
\[
  \tilde{S}_{1,a}=\frac{S_{1,a}}{\sqrt{\frac{1}{2}\dim(\mathcal{C}_1)}\sqrt{\dim^R(\bar{\mathbf{1}})}}\quad\text{and}\quad \tilde{S}_{a,1}=\frac{S_{a,1}}{\sqrt{\frac{1}{2}\dim(\mathcal{C}_1)}\sqrt{\dim^R(\bar{\mathbf{1}})}}.
\]

  Then 
  \begin{multline*}
  (\tilde{S}_{a,1}\tilde{S}_{1,a})^2=\id, (\tilde{S}_{1,a}\tilde{S}_{a,1})^2=\id, (\tilde{S}_{a,1}T_1\tilde{S}_{1,a}T_a^2)^2 = \frac{\tau^{-}(\mathcal{C}_1)}{\tau^{+}(\mathcal{C}_1)\dim^R(\bar{\mathbf{1}})}Q
  \\\text{and}\quad
  (\tilde{S}_{a,1}T_1^{-1}\tilde{S}_{1,a}T_a^{-2})^2 = \frac{\tau^{+}(\mathcal{C}_1)\dim^R(\bar{\mathbf{1}})}{\tau^{-}(\mathcal{C}_1)}Q,
  \end{multline*}
where $Q$ is the permutation matrix of the involution $\bar{\phantom{a}}$ restricted to $[I_a]$.
\end{theoreme}

\begin{proof}
  Theorem \ref{thm:S1gSg1} shows that both $\tilde{S}_{1,a}\tilde{S}_{a,1}$ and $\tilde{S}_{a,1}\tilde{S}_{1,a}$ are signed permutation matrices and are of order $2$.

  It follows from the proposition \ref{prop:S-and-T} that
\[
  (\tilde{S}_{a,1}T_1\tilde{S}_{1,a}T_a^2)^2 = \frac{\left(\frac{1}{2}\tau^{-}(\mathcal{C}_1)\right)^2}{\frac{1}{2}\dim(\mathcal{C}_1)\dim^R(\bar{\mathbf{1}})}Q,
\]
since $S_{a,a}^2=\frac{1}{2}\dim(\mathcal{C}_1)\dim^R(\bar{\mathbf{1}})Q$ and $QT_{a}=T_aQ$.

Finally, for the last relation, we start by taking the inverse of \eqref{eq:S-and-T}:
\[
  \tilde{S}_{1,a}^{-1}T_1^{-1}\tilde{S}_{a,1}^{-1} = \frac{\frac{1}{2}\dim(\mathcal{C}_1)\dim^R(\bar{\mathbf{1}})}{\frac{1}{2}\tau^-(\mathcal{C}_1)} T_aS_{a,a}^{-1}T_a.
\]

But, thanks to Theorem \ref{thm:S1gSg1}, we have $\tilde{S}_{1,a}^{-1}=\tilde{S}_{a,1}P,\tilde{S}_{a,1}^{-1}=P\tilde{S}_{1,a}$ and $\frac{1}{2}\dim(\mathcal{C}_1)\dim^R(\bar{\mathbf{1}})S_{a,a}^{-1}=S_{a,a}$. Since $PT_1^{-1}P = T_1^{-1}$, we find
\[
  (\tilde{S}_{a,1}T_1^{-1}\tilde{S}_{1,a}T_a^{-2})^2 =  \frac{\frac{1}{2}\dim(\mathcal{C}_1)\dim^R(\bar{\mathbf{1}})}{\left(\frac{1}{2}\tau^-(\mathcal{C}_1)\right)^2}Q.
\]
We conclude using Lemma \ref{lem:tau+tau-}.
\end{proof}

%%% Local Variables:
%%% mode: latex
%%% TeX-master: "article"
%%% End:

\section{Twisted modular data associated to dihedral groups}
\label{sec:twisted-modular-dihedral}

In this section, we start by reviewing the category related to the exotic Fourier transform associated to dihedral groups by Lusztig \cite{lusztig-exotic}. We will work out in particular the details of \cite[3.8]{lusztig-exotic}. We then apply the construction of Section \ref{sec:categ-cont-rep-g} and show that it gives rise to the Fourier matrix of the big family of unipotent characters of twisted dihedral groups, as described in \cite{geck-malle}. We work over the complex field and we fix $d>2$ an integer, $\zeta=\exp(\pi i/d)$ a primitive $2d$-th root of unity and $\xi=\zeta^2$.

\subsection{A Fourier matrix for diedral group}
\label{sec:fourier-dihedral}

We recall the definition of the dihedral modular datum of Lusztig \cite[Section 3.1]{lusztig-exotic}. Denote by $I$ the set of pairs $(i,j)$ such that
\[
  0<i<j<i+j<d\quad\text{or}\quad0=i<j<\frac{d}{2}.
\]
Let $I'$ be the set consisting of $(0,d/2)'$ and $(0,d/2)''$ if $d$ is even, and let $I'=\varnothing$ if $d$ is odd. Let $X$ be the disjoint union of $I$ and $I'$. Following Lusztig, define the matrix $S=(\{x,x'\})_{x,x'\in X}$ by
\[
  \{(i,j),(k,l)\}=\frac{\xi^{il+jk}+\xi^{-il-jk}-\xi^{ik+jl}-\xi^{-ik-jl}}{d}
\] 
if $(i,j)$ and $(k,l)$ belong to $I$,
\[
  \{(i,j),(0,d/2)'\}=\{(i,j),(0,d/2)''\}=\{(0,d/2)',(i,j)\}=\{(0,d/2)'',(i,j)\}=\frac{(-1)^i-(-1)^j}{d}
\]
if $(i,j)$ belongs to $I$ and $d$ is even, and
\begin{align*}
  \{(0,d/2)',(0,d/2)'\}=\{(0,d/2)'',(0,d/2)''\}&=\frac{1-(-1)^{d/2}+d}{2d},\\
  \{(0,d/2)',(0,d/2)''\}=\{(0,d/2)'',(0,d/2)'\}&=\frac{1-(-1)^{d/2}-d}{2d}.
\end{align*}
 
Lusztig also defines a vector $t=(t_x)_{x\in X}$ by $t_{i,j}=\xi^{-ij}$ if $(i,j)\in I$ and $t_x=1$ if $x\in I'$.

\begin{remark}
  \label{rk:modification-lusztig}
  Our formula for the definition of $S$ is not exactly the one of
  Lusztig. To recover the one of Lusztig, it suffices to apply the
  involution on $X$ defined by $(i,j)^{\flat}= (i,d-j)$
  if $(i,j)\in I$ and $i>0$ and $x^{\flat}x$ otherwise: Lusztig's $S$-matrix is $S_{x,x'^{\flat}}$.
\end{remark}

\begin{proposition}[{\cite[Proposition 3.2]{lusztig-exotic}}]
  \label{prop:luzstig-dihedral-decat}
  The matrix $S$ together with the diagonal matrix $T$ with entries $t_x$ for $x\in X$ satisfy
  \[ 
    S^2 = 1 \quad \text{and} \quad (ST)^3=1.
  \]
  Moreover, for any $x,y,z \in X$ the complex number
\[
  \sum_{u\in X}\frac{S_{x,u}S_{y,u}\overline{S_{z,u}}}{S_{(0,1),u}}
\]
is a non-negative integer.
\end{proposition}

The matrix $S$ is called the Fourier matrix and the entries of the matrix $T$ is the diagonal matrix of eigenvalues of the Frobenius.

\subsection{A Fourier matrix for dihedral groups with automorphism}
\label{sec:fourier-twisted-dihedral}

We recall the definition of the modular datum for dihedral groups with automorphism as defined in \cite[6C]{malle-unipotente} and \cite[6.1]{geck-malle}. Let $J$ be the set of pairs $(k,l)$ of odd integers satisfying
\[
  0<k<l<k+l<2d.
\]
The Fourier matrix takes the form $S^{\mathrm{tw}}=(\langle(i,j),(k,l)\rangle)_{(i,j)\in I,(k,l)\in J}$ with
\[
  \langle(i,j),(k,l)\rangle=\frac{\zeta^{il+jk}+\zeta^{-il-jk}-\zeta^{ik+jl}-\zeta^{-ik-jl}}{d}.
\]
There is a notion of ``unipotent character'' indexed by $J$ and the eigenvalue of the Frobenius on the ``unipotent character'' associated to $(k,l)\in J$ is $\zeta^{kl}$. We denote by $F_1$ the diagonal matrix with entries $t_x$ for $x\in I$ and by $F_2$ the diagonal matrix with entries $\zeta^{kl}$ for $(k,l)\in J$.

\begin{proposition}[{\cite[Theorem 6.9]{geck-malle}}]
  The matrices $S^{\mathrm{tw}}$, $F_1$ and $F_2$ satisfy
  \[
    S^{\mathrm{tw}}{}^t\!S^{\mathrm{tw}}=1={}^t\!S^{\mathrm{tw}}S^{\mathrm{tw}}\quad\text{and}\quad (F_2{}^tS^{\mathrm{tw}}F_1^{-1}S^{\mathrm{tw}})^2=1.
  \]
Moreover, for any $x,y,z \in I$ the complex number
\[
  \sum_{u\in J}\frac{S_{x,u}S_{y,u}\overline{S_{z,u}}}{S_{(0,1),u}}
\]
is an integer.
\end{proposition}

\subsection{The Drinfeld center of tilting modules for quantum $\mathfrak{sl}_2$}
\label{sec:WZW}

Let $\mathcal{C}_d$ be the fusion category of tilting modules for $\mathcal{U}_{\zeta}(\mathfrak{sl}_2)$ as defined in \cite[Section 3.3]{bakalov-kirillov}. This fusion category can also be described in term of representations of level $d-2$ of an affine Lie algebra of type $A_1$ with a truncated tensor product. An equivalence between these categories has been proved by Finkelberg \cite{finkelberg}. This category has $d-1$ simple objects $V_1,\ldots,V_{d-1}$, $V_1$ being the unit object. The $S$-matrix of this category is given by
\[
  S_{V_i,V_j}=\frac{\zeta^{ij}-\zeta^{-ij}}{\zeta-\zeta^{-1}}.
\]
This category is pivotal and the corresponding twist is given by
\[
  \theta_{V_i} = \zeta^{(i^2-1)/2}.
\]
In this category, every object is self-dual, and the pivotal structure is hence spherical.

To explain a categorification of its dihedral modular datum, Lusztig consider a degenerate subcategory of the tensor category $\mathcal{C}=\mathcal{C}_d\boxtimes\mathcal{C}_d^{\rev}$. Let us remark that $\mathcal{C}\simeq \mathcal{Z}(\mathcal{C}_d)$ where $\mathcal{Z}$ denotes the Drinfeld center. Indeed, the category $\mathcal{C}_d$ is non-degenerate and therefore is factorizable \cite[Proposition 8.20.12]{egno}: its Drinfeld center is equivalent to $\mathcal{C}_d\boxtimes\mathcal{C}_d^{\rev}$. This category has its $S$-matrix given by
\[
  S_{V_i\boxtimes V_j,V_k\boxtimes V_l}=\frac{\zeta^{ik}-\zeta^{-ik}}{\zeta-\zeta^{-1}}\frac{\zeta^{jl}-\zeta^{-jl}}{\zeta-\zeta^{-1}}
\]
and the twist is given by
\[
  \theta_{V_i\boxtimes V_j}=\zeta^{(i^2-j^2)/2}.
\]

Let $\varepsilon$ be the simple object $V_{d-1}\boxtimes V_{d-1}$, which is of quantum dimension $1$. As $V_{d-1}$ is of square $\mathbf{1}$ in $\mathcal{C}_d$, the object $\varepsilon$ is of square $\mathbf{1}$ in $\mathcal{C}$. The category $\mathcal{C}$ is non-degenerate and the object $\varepsilon$ generates a subcategory isomorphic to $\Rep(\mathbb{Z}/2\mathbb{Z})$. We now apply the results of Section \ref{sec:categ-cont-rep-g}. The simple object $X=V_i\boxtimes V_j$ is in degree $0$ if and only if $S_{\varepsilon,X}=\dim(X)$ and is in degree $1$ if and only if $S_{\varepsilon,X}=-\dim(X)$. As
\[
  S_{\varepsilon,V_i\boxtimes V_j}=(-1)^{i+j}\dim(V_i\boxtimes V_j),
\]
the simple objects of $\mathcal{C}_0$ are of the form $V_i\boxtimes V_j$ with $i\equiv j\mod 2$ and the simple objects of $\mathcal{C}_1$ are of the form $V_i\boxtimes V_j$ with $i\not\equiv j\mod 2$. 

\begin{remark}
  \label{rk:category-luzstig}
  The component of degree $0$ is not exactly the degenerated subcategory
  $\mathcal{C}'$ considered by Lusztig in \cite[Section
  3.8]{lusztig-exotic}. Indeed, he starts with the Deligne tensor product $\mathcal{C}\boxtimes \mathcal{C}$ instead of $\mathcal{C}\boxtimes\mathcal{C}^{\mathrm{rev}}$. He then only obtains the Fourier matrix but not the eigenvalues of the Frobenius.
\end{remark}

As $V_{d-1}\otimes V_i \simeq V_{d-i}$ in $\mathcal{C}_d$, we have $\varepsilon\otimes(V_i\boxtimes V_j)= V_{d-i}\boxtimes V_{d-j}$ in $\mathcal{C}$. Therefore tensorization by $\varepsilon$ on the set $I_0$ of simple objects of $\mathcal{C}_0$ has a fixed point if and only if $d$ is even, namely $V_{d/2}\boxtimes V_{d/2}$.

\subsection{Modularization of the component of degree $0$}
\label{sec:degree_0}

The only non-trivial object in the symmetric center of $\mathcal{C}_0$ is $\varepsilon$, which is of quantum dimension $1$ and of twist $1$. By a result of \cite{bruguieres}, there exists a unique minimal modularization $\mathcal{C}_0^{\text{mod}}$ of the category $\mathcal{C}_0$ together with a braided tensor functor $F\colon\mathcal{C}_0\rightarrow\mathcal{C}_0^{\text{mod}}$. It can be obtained by first adding an isomorphism between $\varepsilon$ and $\mathbf{1}$ and then by taking the idempotent completion. If $d$ is odd, the set of simple objects of $\mathcal{C}_0^{\text{mod}}$ are given by $V_i\boxtimes V_j$ where $i$ and $j$ are of same parity and $0<i<j<d$ or $0<i=j<d/2$. If $d$ is even, there are two more simple objects $(V_{d/2}\boxtimes V_{d/2})_+$ and $(V_{d/2}\boxtimes V_{d/2})_-$. We denote by $\tilde{I}$ the set of pair of integers $(i,j)$ such that $i$ and $j$ are of same parity and $0<i<j<d$ or $0<i=j<d/2$ and by $\tilde{I}'$ the empty set if $d$ is odd, the set containing two elements $(d/2,d/2)_+$ and $(d/2,d/2)_-$ if $d$ is even. Let $\tilde{X}$ be the union of $\tilde{I}$ and $\tilde{I'}$. We will index the $S$-matrix as well as the twist of $\mathcal{C}_0^{\text{mod}}$ by $\tilde{X}$.

\begin{proposition}
  \label{prop:computation-ST-dihedral}
  If $d$ is odd, the $S$-matrix and the values of the twists of the category $\mathcal{C}_0^{\text{mod}}$ are given by
  \[
    S_{(i,j),(k,l)} = \frac{\zeta^{ik}-\zeta^{-ik}}{\zeta-\zeta^{-1}}\frac{\zeta^{jl}-\zeta^{-jl}}{\zeta-\zeta^{-1}}\quad\text{and}\quad \theta_{(i,j)}=\zeta^{(i^2-j^2)/2}.
  \]
  
If $d$ is even, the $S$-matrix and the values of the twists of the category $\mathcal{C}_0^{\text{mod}}$ are given by
  \[
  S_{(i,j),(k,l)} = \frac{\zeta^{ik}-\zeta^{-ik}}{\zeta-\zeta^{-1}}\frac{\zeta^{jl}-\zeta^{-jl}}{\zeta-\zeta^{-1}},
  \]
  \begin{multline*}
    S_{(i,j),(d/2,d/2)_+}=S_{(i,j),(d/2,d/2)_-}=S_{(d/2,d/2)_+,(i,j)}=\\S_{(d/2,d/2)_-,(i,j)}=\frac{1}{2}\frac{\zeta^{id/2}-\zeta^{-id/2}}{\zeta-\zeta^{-1}}\frac{\zeta^{jd/2}-\zeta^{-jd/2}}{\zeta-\zeta^{-1}},
  \end{multline*}
  \begin{align*}
    S_{(d/2,d/2)_+,(d/2,d/2)_+}=S_{(d/2,d/2)_-,(d/2,d/2)_-}
    &=\frac{1}{2(\zeta-\zeta^{-1})^2}\left((-1)^{d/2}-1-d\right),\\
    S_{(d/2,d/2)_+,(d/2,d/2)_-}=S_{(d/2,d/2)_-,(d/2,d/2)_+}
    &=\frac{1}{2(\zeta-\zeta^{-1})^2}\left((-1)^{d/2}-1+d\right)
  \end{align*}
and
  \[
  \theta_{(i,j)}=\zeta^{(i^2-j^2)/2}\quad \text{and}\quad \theta_{(d/2,d/2)_+}=\theta_{(d/2,d/2)_-}=1.
  \]
\end{proposition}

\begin{proof}
  We recall the construction of Müger \cite{mueger} of the modularization of $\mathcal{C}_0$. Fix $\varphi$ an isomorphism between $\varepsilon\otimes \varepsilon$ and $\mathbf{1}$. Let $\mathcal{D}$ be the category with the same objects as $\mathcal{C}_0$ and with space of morphisms
\[
  \Hom_{\mathcal{D}}(X,Y):=\Hom_{\mathcal{C}_0}(X,Y)\oplus\Hom_{\mathcal{C}_0}(X,\varepsilon\otimes Y).
\]
The composition between two morphisms is defined the obvious way, using $\varphi$ if necessary. The tensor product is the same as the one in $\mathcal{C}_0$ on the objects and the tensor product of two morphisms is defined the obvious way, using the braiding and $\varphi$ if necessary. Duality in $\mathcal{C}_0$ naturally extends to a duality in $\mathcal{D}$, so does the pivotal structure. Note that the trace of a morphism $f\in\End_{\mathcal{D}}(X)$ coming from a morphism $X\rightarrow\varepsilon\otimes X$ in $\mathcal{C}_0$ is an element of $\Hom_{\mathcal{D}}(\mathbf{1},\mathbf{1})$ coming from a morphism $\mathbf{1}\rightarrow \varepsilon\otimes \mathbf{1}$ in $\mathcal{C}_0$ which is necessarily $0$.

The modularization $\mathcal{C}_0^{\text{mod}}$ of $\mathcal{C}_0$ is then the idempotent completion of $\mathcal{D}$: its objects are pairs $(X,e)$ where $X$ is on object of $\mathcal{D}$ and $e\in\Hom_{\mathcal{D}}(X,X)$ is an idempotent; a morphism between $(X,e)$ and $(Y,f)$ is simply a morphism $g\colon X\rightarrow Y$ such that $f\circ g = g\circ e$. All the structures extend from $\mathcal{D}$ to its idempotent completion. Moreover for $f\in\End_{\mathcal{C}_0^{\text{mod}}}((X,e))$ one have
\[
  \Tr_{(X,e)}^{\mathcal{C}_0^{\text{mod}}}(f)=\Tr_{X}^{\mathcal{D}}(f\circ e).
\]
Note that we have denoted in upperscript the category in which we compute the trace (here, the pivotal structure is spherical, so we droppel the upperscript $R$ ou $L$). The functor $F$ is simply $X\mapsto (X,\id_X)$ and therefore $F(X)$ is simple if and only if $X$ is simple and $X\not\simeq \varepsilon\otimes X$. If $X$ is simple and $X\simeq \varepsilon\otimes X$, there exists an isomorphism $\gamma\in\Hom_{\mathcal{D}}(X,X)$ such that $\gamma\circ\gamma=\id_X$ arising from an suitable isomorphism $g\colon X\rightarrow \varepsilon\otimes X$. Hence $e_{\pm}=\frac{1}{2}(\id_X\pm\gamma)$ is an idempotent and $F(X) = X_{+}\oplus X_{-}$ where $X_{\pm}=(X,e_{\pm})$.

Now suppose that $X$ is a simple object in $\mathcal{C}_0$ such that $F(X)=X_{+}\oplus X_{-}$ and that $Y$ is a simple object in $\mathcal{C}_0$ such that $F(Y)$ is still simple in $\mathcal{C}_0^{\text{mod}}$. Then
\[
  \Tr_{F(Y)\otimes X_{\pm}}^{\mathcal{C}_0^{\text{mod}}}(c_{X_{\pm},F(Y)}\circ c_{F(Y),X_{\pm}}) = \Tr_{Y,X}^{\mathcal{D}}(c_{X,Y}\circ c_{Y,X}\circ \id\otimes e_{\pm}) = \frac{1}{2}\Tr_{Y,X}^{\mathcal{C}_0}(c_{Y,X}\circ c_{X,Y}),
\]
since the morphism $c_{Y,X}\circ c_{X,Y} \circ \id_Y\otimes \gamma$ in $\mathcal{D}$ comes from a morphism $X\otimes Y\rightarrow \varepsilon\otimes X \otimes Y$ in $\mathcal{C}_0$ whose trace is zero. Similarly,
\[
  \Tr_{X_+\otimes X_{\pm}}^{\mathcal{C}_0^{\text{mod}}}(c_{X_{\pm},X_+}\circ c_{X_+,X_{\pm}})=\frac{1}{4}\Tr_{X,X}^{\mathcal{C}_0}(c_{X,X}\circ c_{X,X}\circ(\id_{X\otimes X}\pm (\varphi\otimes\id_{X\otimes X})\circ(\id_{\varepsilon}\otimes c_{X,\varepsilon}\otimes\id_X)\circ(g\otimes g))
\]
and
\[
  \Tr_{X_-\otimes X_{\pm}}^{\mathcal{C}_0^{\text{mod}}}(c_{X_{\pm},X_-}\circ c_{X_-,X_{\pm}})=\frac{1}{4}\Tr_{X,X}^{\mathcal{C}_0}(c_{X,X}\circ c_{X,X}\circ(\id_{X\otimes X}\mp (\varphi\otimes\id_{X\otimes X})\circ(\id_{\varepsilon}\otimes c_{X,\varepsilon}\otimes\id_X)\circ(g\otimes g)).
\]

If $d$ is odd, there are no simple objects $X$ in $\mathcal{C}_0$ such that $X\simeq \varepsilon\otimes X$. Therefore the $S$-matrix of $\mathcal{C}_0^{\text{mod}}$ is simply a submatrix of the $S$-matrix of $\mathcal{C}_0$.

If $d$ is even, there is exactly one simple object $X$ such that $X\simeq \varepsilon\otimes X$. The $S$-matrix has then the following form
\[
  S^{\mathcal{C}_0^{\text{mod}}} = 
  \begin{pmatrix}
    s &\hspace*{-\arraycolsep}\vline\hspace*{-\arraycolsep} & {}^tl & {}^tl\\
    \hline
    l &\hspace*{-\arraycolsep}\vline\hspace*{-\arraycolsep} & \alpha & \beta\\
    l &\hspace*{-\arraycolsep}\vline\hspace*{-\arraycolsep} & \beta & \alpha
  \end{pmatrix}
\]
where $s$ is a square matrix, $l$ a row vector and $\alpha,\beta$ are scalars. We have arranged the simple objects in an order such that $X_+$ and $X_-$ are the two last ones. The only unknown values in this matrix are $\alpha$ and $\beta$. We have $2(\alpha+\beta)=S^{\mathcal{C}_0}_{X,X}$. Now, the relation $S^2 = \dim(\mathcal{C}_0^{\text{mod}})E$, where $E$ is the permutation matrix giving the duality, gives
\[
  \begin{pmatrix}
    s^2+2{}^tll &\hspace*{-\arraycolsep}\vline\hspace*{-\arraycolsep} & s{}^tl+(\alpha+\beta){}^tl & s{}^tl+(\alpha+\beta){}^tl\\
    \hline
    ls+(\alpha+\beta)l &\hspace*{-\arraycolsep}\vline\hspace*{-\arraycolsep} & l{}^tl+\alpha^2+\beta^2 & l{}^tl+2\alpha\beta\\
    ls+(\alpha+\beta)l &\hspace*{-\arraycolsep}\vline\hspace*{-\arraycolsep} & l{}^tl+2\alpha\beta & l{}^tl+\alpha^2+\beta^2
  \end{pmatrix}
  =\dim(\mathcal{C}_0^{\text{mod}})
  \begin{pmatrix}
    1 &\hspace*{-\arraycolsep}\vline\hspace*{-\arraycolsep} & 0 & 0\\
    \hline
    0 &\hspace*{-\arraycolsep}\vline\hspace*{-\arraycolsep} & \gamma & \delta\\
    0 &\hspace*{-\arraycolsep}\vline\hspace*{-\arraycolsep} & \eta & \nu 
  \end{pmatrix},
\]
where $
\begin{pmatrix}
  \gamma & \delta\\
  \eta & \nu
\end{pmatrix}$
is the identity or $
\begin{pmatrix}
  0 & 1\\
  1 & 0
\end{pmatrix}$
whether $X_+$ and $X_{-}$ are auto-dual or dual to each other. Therefore, $(\alpha-\beta)^2=\dim(\mathcal{C}_0^{\text{mod}})$ if $X_+$ and $X_-$ are auto dual, and $(\alpha-\beta)^2=-\dim(\mathcal{C}_0^{\text{mod}})$ if $X_+$ and $X_-$ are dual to each other, so that in any cases $(\alpha-\beta)^4=\dim(\mathcal{C}_0^{\text{mod}})^2$.

Denoting by $T$ the diagonal matrix with entries the action of the inverse of the twist on simple objects of $\mathcal{C}_0^{\text{mod}}$, we have $(ST^{-1})^3=\dim(\mathcal{C}_0^{\text{mod}})\tau^{+}(\mathcal{C}_0^{\text{mod}})1$. By explicitly computing $(ST^{-1})^3$, we obtain that $\theta_{X}^{-3}(\alpha-\beta)^3=\dim(\mathcal{C}_0^{\text{mod}})\tau^{+}(\mathcal{C}_0^{\text{mod}})$. Hence $\alpha-\beta=\theta_X^3\tau^{-}(\mathcal{C}_0^{\text{mod}})$. Finally
\[
  \alpha=\frac{1}{4}\left(S^{\mathcal{C}_0}_{X,X}+2\theta_X^3\tau^{-}(\mathcal{C}_0^{\text{mod}})\right)\quad\text{and}\quad\beta=\frac{1}{4}\left(S^{\mathcal{C}_0}_{X,X}-2\theta_X^3\tau^{-}(\mathcal{C}_0^{\text{mod}})\right).
\]
The computation of $\tau^-(\mathcal{C}_0^{\text{mod}})$ requires the knowledge of Gauss sums, but we can avoid this technical step. We notice that $2\tau^+(\mathcal{C}_0^{\text{mod}})=\tau^{+}(\mathcal{C})$. Indeed, we have
\[
  \tau^{+}(\mathcal{C}) = \sum_{X\in\Irr(\mathcal{C}_0)}|X|^2\theta_X+ \sum_{X\in\Irr(\mathcal{C}_1)}|X|^2\theta_X
= \sum_{X\in\Irr(\mathcal{C}_0)}|X|^2\theta_X=\tau^+(\mathcal{C}_0)
\]
because $\theta_{\varepsilon\otimes X}=-\theta_{X}$ for
$X\in\Irr(\mathcal{C}_1)$ and because tensorization by $\varepsilon$
has no fixed points on $\Irr(\mathcal{C}_1)$. But
$\tau^{\pm}(\mathcal{C}_0)=2\tau^{\pm}(\mathcal{C}_0^{\text{mod}})$. Now, since $\mathcal{C}$ is equivalent to of $\mathcal{C}_d\boxtimes \mathcal{C}_d^{\mathrm{rev}}$, we have $\tau^{+}(\mathcal{C})=\tau^{+}(\mathcal{C}_d)\tau^{-}(\mathcal{C}_d)=\dim(\mathcal{C}_d)$
and
\[
  \dim(\mathcal{C}_d)=\sum_{i=1}^{d-1}\left(\frac{\zeta^i-\zeta^{-i}}{\zeta-\zeta^{-1}}\right) = -\frac{2d}{(\zeta-\zeta^{-1})^2}.
\]
Finally,
\[
  \alpha=\frac{1}{2(\zeta-\zeta^{-1})^2}\left((-1)^{d/2}-1-d\right)\quad\text{and}\quad\beta=\frac{1}{2(\zeta-\zeta^{-1})^2}\left((-1)^{d/2}-1+d\right).
\]
\end{proof}

Now we renormalize the $S$-matrix by the positive square root of $\dim(\mathcal{C}_0^{\text{mod}})$ which is equal to
\[
  \sqrt{\dim(\mathcal{C}_0^{\text{mod}})}=-\frac{d}{(\zeta-\zeta^{-1})^2}.
\]
If $d$ is odd, the renormalized $S$-matrix $\tilde{S}$ is given by
\[
  \tilde{S}_{(i,j),(k,l)} = \frac{\zeta^{ik-jl}+\zeta^{-ik+jl}-\zeta^{ik+jl}-\zeta^{-ik-jl}}{d},
\]
for $i$ and $j$ integers of same parity such that $0<i<j<d$ or $0<i=j<d/2$. The renormalized $S$-matrix $\tilde{S}$ is given by
\[
  \tilde{S}_{(i,j),(k,l)} = \frac{\zeta^{ik-jl}+\zeta^{-ik+jl}-\zeta^{ik+jl}-\zeta^{-ik-jl}}{d},
\]
for $(i,j),(k,l)\in\tilde{I}$, and if $d$ is even
\[
  \tilde{S}_{(i,j),(d/2,d/2)_+}
  =\tilde{S}_{(d/2,d/2)_+,(i,j)}
  =\tilde{S}_{(i,j),(d/2,d/2)_-}
  =\tilde{S}_{(d/2,d/2)_-,(i,j)}
  =\frac{1}{d}\left((-1)^{(i-j)/2}-(-1)^{(i+j)/2}\right),
\]
for $(i,j)\in\tilde{I}$, and if $d$ is even
\begin{align*}
  \tilde{S}_{(d/2,d/2)_+,(d/2,d/2)_+}=\tilde{S}_{(d/2,d/2)_-,(d/2,d/2)_-}&=\frac{1}{2d}\left(1-(-1)^{d/2}+d\right)\\
  \tilde{S}_{(d/2,d/2)_+,(d/2,d/2)_-}=\tilde{S}_{(d/2,d/2)_-,(d/2,d/2)_+}&=\frac{1}{2d}\left(1-(-1)^{d/2}-d\right).
\end{align*}

Let $\Psi_0\colon X \rightarrow \tilde{X}$ sending $(i,j)\in I$ to $(j-i,i+j)$ and, if $d$ is even, $(0,d/2)'$ on $(d/2,d/2)_+$ and $(0,d/2)''$ on $(d/2,d/2)_-$. It is easily checked that $\Psi_0$ is a bijection, and its inverse on $\tilde{I}$ is given by $(i,j)\mapsto \left(\frac{j-i}{2},\frac{i+j}{2}\right)$. The following for the matrix $\tilde{S}$ is due to Lusztig \cite[3.8]{lusztig-exotic}.

\begin{theoreme}
  \label{thm:cat-dihedral}
  The category $\mathcal{C}_0^{\text{mod}}$ is a categorification of the modular datum associated to the dihedral group. More explicitly, if $\tilde{S}$ denotes the renormalized $S$-matrix of $\mathcal{C}_0^{\text{mod}}$ and $\theta$ the twist,
\[
  \tilde{S}_{\Psi_0(x),\Psi_0(x')}=\{x,x'\}\quad\text{and}\quad \theta_{\Psi_0(x)}=t_x,
\] 
for every $x,x'\in X$.
\end{theoreme}

\subsection{The degree $1$ part of $\mathcal{C}$}
\label{sec:degree-1}

We now turn to the study of the matrix extracted from the degree $1$
part of $\mathcal{C}$ and show that we recover the Fourier matrix
associated to dihedral groups with non-trivial automorphism, as
described in Section \ref{sec:fourier-twisted-dihedral}. The simple
objects of the category $\mathcal{C}_1$ are $V_k\boxtimes V_l$ with
$0< k,l < d$ and $k\not\equiv l \mod 2$ and we have
$\varepsilon \otimes V_k\boxtimes V_l =V_{d-k}\boxtimes V_{d-l}$. We
then choose $\{V_k\boxtimes V_l \vert 0 < k < l < d, k\not\equiv l
\mod 2\}$ as representatives of the orbits on $\Irr(\mathcal{C}_1)$
under tensorization by $\varepsilon$. Denote by $\tilde{J}$ the set of
pairs of integers $(k,l)$ such that $0<k<l<d$ and $k\not\equiv l \mod 2$.  

The matrix $S_{0,1}$ (resp. $S_{1,0}$) is indexed by
$\tilde{I}\times\tilde{J}$ (resp. $\tilde{J}\times\tilde{I}$) and
\[
  (S_{0,1})_{(i,j),(k,l)} =
  \frac{\zeta^{ik}-\zeta^{-ik}}{\zeta-\zeta^{-1}}\frac{\zeta^{jl}-\zeta^{-jl}}{\zeta-\zeta^{-1}},
\]
for any $(i,j)\in\tilde{I}$ and $(k,l)\in\tilde{J}$.

Let $\Psi_1\colon J \rightarrow \tilde{J}$ sending $(k,l)\in I$ to
$\left(\frac{l-k}{2},\frac{k+l}{2}\right)$. It is easily checked that
$\Psi_1$ is a bijection, and its inverse is given by $(k,l)\mapsto
(l-k,k+l)$. An easy computation shows the following.

\begin{theoreme}
  \label{thm:cat-twisted-dihedral}
  The renormalized $S$-matrix $\tilde{S}_{0,1}$ associated to
  $\mathcal{C}$ is the Fourier matrix
  associated to the non trivial family of ``unipotent characters'' of
  the dihedral group with non-trivial automorphism:
  \[
    (\tilde{S}_{0,1})_{\Psi_0((i,j)),\Psi_1((k,l))}=\langle(i,j),(k,l)\rangle.
  \]
  Moreover, the inverse of the eigenvalue of the Frobenius is given by the square of
  the twist on $\mathcal{C}_1$:
  \[
    \theta_{\Psi_{1}((k,l))}^{-2}=\zeta^{kl}
  \]
\end{theoreme}

Therefore the matrix $F_1$ is equal to the matrix $T_0^{-1}$ and the
matrix $F_2$ is the matrix $T_1^2$. The relation \cite[Theorem 6.9,
(F6')]{geck-malle} translates into the relation
\[
  (T_1^2\tilde{S}_{10}T_0\tilde{S}_{01})^2=1,
\]
and we have proven that in general
\[
  (T_1^{-2}\tilde{S}_{10}T_0^{-1}\tilde{S}_{01})^2=\frac{\tau^{+}(\mathcal{C}_0)}{\tau^{-}(\mathcal{C}_0)}\dim^R(\bar{\mathbf{1}})\tilde{S}_{10}\tilde{S}_{01}.
\]
But in this setting, the category $\mathcal{C}$ is spherical, hence
$\bar{\mathbf{1}}=\mathbf{1}$, we have seen that
$\tau^+(\mathcal{C}_0)=\tau^{-}(\mathcal{C}_0)$ and every object is
auto-dual, so that $\tilde{S}_{1,0}\tilde{S}_{0,1}=1$. Finally noticing
that the entries of $\tilde{S}_{0,1}$ and $\tilde{S}_{1,0}$ are real
shows that the two relations are indeed equivalent.

%%% Local Variables:
%%% mode: latex
%%% TeX-master: "article"
%%% End:

\section{Drinfeld double of a central extension of a finite group and ${}^2\!F_4$}
\label{sec:central-extension}

In this section, we show that the modular category of representations of the Drinfeld double of a central extension by a cyclic group fits into the framework of Section \ref{sec:categ-cont-rep-g}. If we consider a central extension of the symmetric group $\mathfrak{S}_4$ we then recover the Fourier matrix of the big family for the Ree group ${}^2\!F_4$, which was defined in \cite{geck-malle} and for which no categorical explanation was known.

\subsection{Reminders on the Drinfeld double of a finite group}
\label{sec:drinfeld-double}

The Drinfeld double of a finite group $G$ is a special case of a more general construction for finite dimensional Hopf algebras \cite[Chapter IX]{kassel}. Let $\mathbb{C}G$ be the group algebra of $G$ and $\mathbb{C}[G]$ the algebra of $\mathbb{C}$-valued functions on $G$. We denote by $g\in \mathbb{C}G$ the element corresponding to $g\in G$ and by $e_g\in\mathbb{C}[G]$ the function such that $e_g(h) = \delta_{g,h}$. The Drinfeld double $D(G)$ is isomorphic, as a vector space, to the tensor product $\mathbb{C}[G]\otimes \mathbb{C}G$ and $(e_gh)_{g,h\in G}$ is therefore a basis of $D(G)$. On this basis, the multiplication is defined by
\[
  (e_{g}h)(e_{g'}h') = \delta_{g,hg'h^{-1}}e_{g}(hh').
\]
The unit is $\sum_{g\in G}e_g$. This algebra is also a Hopf algebra, its coproduct $\Delta$, its counit $\varepsilon$ and its antipode $S$ are given by
\[
  \Delta(e_gh) = \sum_{g_1g_2=g}e_{g_1}h\otimes e_{g_2}h,\quad \varepsilon(e_gh) = \delta_{g,1}\quad\text{and}\quad S(e_gh) = h^{-1}e_{g^{-1}} = e_{h^{-1}g^{-1}h}h^{-1}.
\]
The category $D(G)\text{-}\mathrm{mod}$ of finite dimensional $D(G)$-modules is a fusion category \cite[Section 3.2]{bakalov-kirillov}: it is in particular semisimple and has a finite number of simple objects.

The algebra $D(G)$, as any Drinfeld double, has a universal $R$-matrix given by
\[
  R = \sum_{g\in G}g\otimes e_g,
\]
which endows the category $D(G)\text{-}\mathrm{mod}$ with a braiding \cite[Proposition XIII.1.4]{kassel}. The square of the antipode is the identity, hence the usual identification of a vector space with its bidual defines a pivotal structure, which is moreover spherical since the quantum dimensions are nothing more than the usual dimensions.

The classification of simple modules is given in \cite[Section 3.2]{bakalov-kirillov}. They are in bijection with the pairs $(g,\rho)$, where $g\in G$ and $\rho$ is an irreducible representation of the centralizer $Z_G(g)$ of $g$ in $G$, modulo the equivalence relation given by $(g,\rho)\sim (hgh^{-1},^{h}\!\rho)$. Here, ${}^h\!\rho$ the representation of $hZ_G(g)h^{-1} = Z_{G}(hgh^{-1})$ defined by ${}^{h}\!\rho(hxh^{-1}) = \rho(x)$. We will denote the simple representation of $D(G)$ corresponding to $(g,\rho)$ by $V_{g,\rho}$. Finally, it is not difficult to compute the $S$-matrix of this category \cite[Section 3.2]{bakalov-kirillov} and the category $D(G)\text{-}\mathrm{mod}$ is modular: its $S$-matrix is invertible or equivalently the only simple object in its symmetric center is $\mathbf{1}$.

\subsection{Drinfeld double of a central extension}
\label{sec:drinfeld-ext}

We fix here a finite group $G$ and we consider a central extension $\tilde{G}$ of $G$ by a cyclic group $A$ of order $n$. We have an exact sequence
\[
\begin{tikzcd}
  1 \arrow[r] & A \ar[r,"\iota"] & \tilde{G} \ar[r,"\pi"] & G \ar[r] & 1,
\end{tikzcd}
\]
and we will see $A$ as a subgroup of $\tilde{G}$.

The representations $V_{a,1}$ for $a\in A$ generate in $D(\tilde{G})\text{-}\mathrm{mod}$ a subcategory isomorphic to $\Rep(\hat{A})$. Then the category of finite dimensional modules over $D(\tilde{G})$ is graded by the group $\hat{A}$. Remark that $V_{a,1}$ is a one dimensional representation of $D(G)$ and that $V_{a,1}\otimes V_{b,1} \simeq V_{ab,1}$ for any $a,b\in A$.

\begin{lemma}
  \label{lem:graduation-Ahat}
  Let $\alpha\in \hat{A}$. For $g\in\tilde{G}$ and $\rho$ an irreducible representation of $Z_{\tilde{G}}(g)$, the simple module $V_{g,\rho}$ is in the component of degree $\alpha$ if and only if $A$ acts on $\rho$ by multiplication by the character $\alpha$.
\end{lemma}

\begin{proof}
  It suffices to compute the double braiding of $V_{g,\rho}$ with any object of the form $V_{a,1}$, $a\in A$. The morphism $c_{V_{a,1},V_{g,\rho}}\colon V_{a,1}\otimes V_{g,\rho}\rightarrow V_{g,\rho}\otimes V_{a,1}$ is given by
\[
  c_{V_{a,1},V_{g,\rho}}(v\otimes w) = \sum_{g\in\tilde{G}}e_g\cdot w \otimes v = w\otimes w
\]
and the morphism $c_{V_{g,\rho},V_{a,1}}\colon V_{g,\rho}\otimes V_{a,1}\rightarrow V_{a,1}\otimes V_{g,\rho}$ is given by
\[
  c_{V_{g,\rho},V_{a,1}}(w\otimes v) = v\otimes a\cdot w,
\]
since $e_g(V_{a,1})=0$ if $g\neq a$. The double braiding is then given by the action of $a$ on $V_{g,\rho}$ as expected.
\end{proof}

In order to extract form the $S$-matrix of $D(\tilde{G})\text{-}\mathrm{mod}$ a crossed $S$-matrix as in Section \ref{sec:categ-cont-rep-g}, we need to understand the fixed points on the set of simple objects of the component of trivial degree, under tensorization by $V_{a,1}$.

\begin{proposition}
  \label{prop:fixed-points-ext}
  Let $g\in\tilde{G}$ and $\rho$ an irreducible representation of $Z_{\tilde{G}}(g)$. The simple object $V_{g,\rho}$ is stable under tensorization by $V_{a,1}$ if and only if the conjugacy class of $g$ and $ag$ are equal and $\rho\simeq {}^{r}\!\rho$, where $r\in\tilde{G}$ is such that $ag=rgr^{-1}$.
\end{proposition}

\begin{proof}
  The simple module $V_{g,\rho}$ is supported by the conjugacy class $[g]$ of $g$ in $\tilde{G}$ and $V_{a,1}\otimes V_{g,\rho}$ is supported by $a[g] = [ag]$ so that $V_{a,1}\otimes V_{g,\rho}\simeq V_{ag,\varphi}$ for $\varphi$ a simple representation of $Z_{\tilde{G}}(ag)=Z_{\tilde{G}}(g)$. It is checked that we have $\varphi\simeq\rho$ and then $V_{a,1}\otimes V_{g,\rho}\simeq V_{ag,\rho}$. But there is an isomorphism between $V_{ag,\rho}$ and $V_{g,\rho}$ if and only if $[g]=[ag]$ and $\rho\simeq ^{r}\!\rho$, where $r\in\tilde{G}$ is such that $ag=rgr^{-1}$.
\end{proof}

\begin{remark}
  \label{rk:fixed-points}
  If $\rho(a)\neq \id$, there is never an isomorphism between $V_{g,\rho}$ and $V_{a,1}\otimes V_{g,\rho}$. Indeed this follows from Lemma \ref{lem:fixed-points-trivial}. We can see it directly using Proposition \ref{prop:fixed-points-ext}: if such an isomorphism exists then we would have $\chi_{\rho}(g)=\chi_{\rho}(ag)$, where $\chi_{\rho}$ is the character of $\rho$, and therefore $\chi_{\rho}(a)=\chi_\rho(1)$ since $a$ is central in $Z_G(g)$.
\end{remark}

Since the category $D(\tilde{G})\text{-}\mathrm{mod}$ is modular, the symmetric center of the trivial component $(D(\tilde{G})\text{-}\mathrm{mod})_{1}$ is equal to $\Rep(\hat{A})$: the only simple objects of this symmetric center are the $V_{a,1}$ for $a\in A$, which are all of dimension $1$ and of twist $1$. By the criterion of Bruguières \cite{bruguieres}, a modularization of this category exists and is unique up to equivalence.

\begin{proposition}
  \label{prop:modularisation-double-extension}
  The modularization of the trivial component of $D(\tilde{G})\text{-}\mathrm{mod}$ is the modular category $D(G)\text{-}\mathrm{mod}$.
\end{proposition}

\begin{proof}
  Let $\mathcal{C}=D(\tilde{G})\text{-}\mathrm{mod}$ and $\mathcal{C}_{1}$ the component of trivial grading. We first define a functor $F\colon \mathcal{C}_0 \rightarrow D(G)\text{-}\mathrm{mod}$ which respect the braiding and the twist. Let $D_{\tilde{G}}(G)$ the Hopf subalgebra of $D(\tilde{G})$ generated by $e_{g}h$ for $g\in\tilde{G}$ and $h\in G$. The three algebras $D(\tilde{G})$, $D_{\tilde{G}}(G)$ and $D(G)$ are related as follows:
\[
\begin{tikzcd}
  & D_{\tilde{G}}(G) & \\
  D(\tilde{G}) \ar[ru,"p"]& & D(G)\ar[lu,"i"']
\end{tikzcd}
\]
where $p(e_gh) = e_g\pi(h)$ for $g,h\in\tilde{G}$ and $i(e_gh) = \sum_{k\in\pi^{-1}(g)}e_kh$ for $g,h \in G$. The application $p$ is surjective and $i$ is injective. Now, a representation of $D(\tilde{G})$ factors through $p$ if and only if $A$ acts trivially on it, which is the case if and only if this representation is in $\mathcal{C}_{1}$. The functor $F$ is defined as the restriction via $i$ of this quotient. Since $p$ and $i$ are Hopf algebra morphisms, the braiding and the ribbon structure are preserved.

The category $D(G)\text{-}\mathrm{mod}$ being modular, it remains to show that the functor $F$ is dominant: every $D(G)$-module is the direct summand of an object in the image of $F$. Let $g\in G$ and $\rho$ a representation of $Z_G(g)$. We first choose $\tilde{g}\in\pi^{-1}(g)$. The quotient map $\pi$ restricts into a group morphism $Z_{\tilde{G}}(\tilde{g})\rightarrow Z_{G}(g)$, which is not necessarily surjective. Let us choose a simple summand $\tilde{\rho}$ of $\rho$ viewed as a representation of $Z_{\tilde{G}}(\tilde{g})$. The subgroup $A$ acts then trivially on $\tilde{\rho}$ and the simple module $V_{g,\rho}$ appears as a direct summand of $F(V_{\tilde{g},\tilde{\rho}})$.
\end{proof}

\subsection{An example related to the Ree group of type ${}^2\!F_4$}
\label{sec:2F4_fourier}

In Section \ref{sec:degree-1}, we gave an explanation for the Fourier matrices associated to the big families of unipotent characters for twisted dihedral groups $^{2}\!I_2(n)$. In \cite{geck-malle}, Geck and Malle proposed a notion of Fourier matrices for Suzuki and Ree groups, which coincide for ${}^2\!B_2$ and ${}^2\!G_2$ with the corresponding cases among dihedral groups. We now consider the Ree group of type ${}^2F_4$ and its family consisting of $13$ unipotent characters. An explanation of the Fourier matrix is given in terms of the Drinfeld double of the binary octahedral group.

\subsubsection{The Fourier matrix}
\label{sec:fourier-matrix}

We start by giving explicitly the matrix associated to the big family of unipotent characters of ${}^2\!F_4$:
\[
  S = \frac{1}{12}\left(
    \begin{smallmatrix}
      3&3&6&-6&-3&-3&-3&-3&-3&-3&0&0&0\\
      0&0&0&-4&4&4&0&0&0&0&4&4&-8\\
      -6&-6&0&-4&-2&-2&0&0&0&0&4&4&4\\
      3&3&6&2&1&1&3&3&3&3&4&4&4\\
      3\sqrt 2&-3\sqrt 2&0&0&-3\sqrt 2&3\sqrt 2&3\sqrt 2&3\sqrt 2&-3\sqrt 2&-3\sqrt 2&0&0&0\\
      3\sqrt 2&-3\sqrt 2&0&0&3\sqrt 2&-3\sqrt 2&-3\sqrt 2&3\sqrt 2&-3\sqrt 2&3\sqrt 2&0&0&0\\
      3\sqrt 2&-3\sqrt 2&0&0&3\sqrt 2&-3\sqrt 2&3\sqrt 2&-3\sqrt 2&3\sqrt 2&-3\sqrt 2&0&0&0\\
      0&0&0&0&0&0&6&-6&-6&6&0&0&0\\
      3&3&-6&-6&-3&-3&3&3&3&3&0&0&0\\
      -3&-3&6&-2&-1&-1&3&3&3&3&-4&-4&-4\\
      0&0&0&-4&4&4&0&0&0&0&4&-8&4\\
      0&0&0&-4&4&4&0&0&0&0&-8&4&4\\
      3\sqrt 2&-3\sqrt 2&0&0&-3\sqrt 2&3\sqrt 2&-3\sqrt 2&-3\sqrt 2&3\sqrt 2&3\sqrt 2&0&0&0
    \end{smallmatrix}\right).
\]
The order of rows and columns is not the one of \cite[Theorem 5.4]{geck-malle} but the one chosen in the package CHEVIE \cite{chevie,chevie-jean} of GAP. As in Section \ref{sec:fourier-twisted-dihedral} we have two diagonal matrices $F_1$ and $F_2$ given by
\[
  F_1=\diag(1, 1, 1, 1, 1, \zeta_4, -\zeta_4, -1, 1, 1, \zeta_3, \zeta_3^2, -1)
\]
and
\[
  F_2=\diag(1, 1, 1, -1, -1, -1, -\zeta_4, \zeta_4, \zeta_4, -\zeta_4, -\zeta_3, -\zeta_3^2, -1).
\]

\begin{proposition}[{\cite[Theorem 6.9]{geck-malle}}]
  \label{thm:2F4-decat}
  The matrix $S$ is unitary and satisfies $(F_2SF_1^{-1}S)^2=\id$. Moreover, for any $i,j,k\in\{1,2,\ldots,13\}$, the number
\[
  N_{i,j}^k:=\sum_{l=1}^{13}\frac{S_{i,l}S_{j,l}\overline{S_{k,l}}}{S_{4,l}}
\]
is a non-negative integer.
\end{proposition}

Thanks to the unitarity of $S$, we can define a free unital and associative $\mathbb{Z}$-algebra with a basis $(b_i)_{1\leq i \leq 13}$ with multiplication given on this basis by
\[
  b_i\cdot b_j := \sum_{k=1}^{13}N_{i,j}^k b_{13},
\]
and the unit element is $b_4$.

\subsubsection{The binary octahedral group and its Drinfeld double}
\label{sec:binary-octahedral}

Using the notation of Section \ref{sec:drinfeld-ext}, we take $G=\mathfrak{S}_4$ and $A=\mathbb{Z}/2\mathbb{Z}$. Since $H^2(G;A)$ is isomorphic to the Klein four-group, there exists four non-isomorphic central extensions of $G$ by $A$. We take for $\tilde{G}$ the binary octahedral group, which has a presentation given by \cite[Theorem 2.8]{hoffman-humphreys}
\[
  \tilde{G}=\left\langle z,t_1,t_2,t_3\ \middle\vert\ z^2=1,t_i^2=z,(t_1t_2)^3=z,(t_2t_3)^3=z,(t_1t_3)^2=z\right\rangle.
\]
In GAP, one can acces to the binary octahedral group using $\verb|SmallGroup(48,28)|$. The element $z$ is central and the quotient map $\pi\colon\tilde{G}\rightarrow G$ sends the element $t_i$ on the transposition $(i, i+1)$. The conjugacy classes are given in \cite[Theorem 3.8]{hoffman-humphreys} and we take as representatives of these classes the following elements:
\[
1,z,t_1,t_1t_3,t_1t_2,zt_1t_2,t_1t_2t_3,zt_1t_2t_3.
\]

In order to compute the simple modules of $D(\tilde{G})$, we need to know the structure of the centralizers of the above elements. Computations with GAP shows that
\begin{multline*}
  Z_{\tilde{G}}(1)=Z_{\tilde{G}}(z)={\tilde{G}},Z_{\tilde{G}}(t_1)\simeq\mathbb{Z}/4\mathbb{Z},Z_{\tilde{G}}(t_1t_3)=\mathbb{Z}/8\mathbb{Z},Z_{\tilde{G}}(t_1t_2)=Z_{\tilde{G}}(zt_1t_2)\simeq\mathbb{Z}/6\mathbb{Z}\\\text{and}\quad Z_{\tilde{G}}(t_1t_2t_3)=Z_{\tilde{G}}(zt_1t_2t_3)\simeq\mathbb{Z}/8\mathbb{Z}.
\end{multline*}

The character table of $\tilde{G}$ is given in \cite[Table 4.7]{hoffman-humphreys} and reproduced in Table \ref{tbl:char-table}. 

\begin{table}[h]
  \centering
\begin{tabular}[h]{c|cccccccc}
  & $1$ & $z$ & $t_1$ & $t_1t_3$ & $t_1t_2$ & $zt_1t_2$ & $t_1t_2t_3$ & $zt_1t_2t_3$\\
\hline
$1$ & $1$ & $1$ & $1$ & $1$ & $1$ & $1$ & $1$ & $1$\\
$\varepsilon$ & $1$ & $1$ & $-1$ & $1$ & $1$ & $1$ & $-1$ & $-1$\\ 
$\chi_2$ & $2$ & $2$ & $0$ & $2$ & $-1$ & $-1$ & $0$ & $0$\\
$\chi_3$ & $3$ & $3$ & $1$ & $-1$ & $0$ & $0$ & $-1$ & $-1$\\ 
$\chi_3'$ & $3$ & $3$ & $-1$ & $-1$ & $0$ & $0$ & $1$ & $1$\\
$\psi_2$ & $2$ & $-2$ & $0$ & $0$ & $1$ & $-1$ & $\sqrt{2}$ & $-\sqrt{2}$\\
$\psi_2'$ & $2$ & $-2$ & $0$ & $0$ & $1$ & $-1$ & $-\sqrt{2}$ & $\sqrt{2}$\\
$\psi_4$ & $4$ & $-4$ & $0$ & $0$ & $-1$ & $1$ & $0$ & $0$ 
\end{tabular}
\caption{Character table of $\tilde{G}$}\label{tbl:char-table}
\end{table}

Except from the centralizers of $1$ and $z$, every other centralizer is cyclic and there is then $56$ irreducible $D(\tilde{G})$-modules. We introduce the following labelling: for $k\in{4,6,8}$, we choose a primitive $k$-th root of unity $\zeta_k$ and we identify $\Irr(\mathbb{Z}/k\mathbb{Z})$ with $\left\{\zeta_k^r\ \middle\vert\ r\in\mathbb{Z}/k\mathbb{Z}\right\}$. Let $\varepsilon$ be the representation $V_{z,1}$, which is of dimension $1$, of twist $1$ and of tensor square isomorphic to the unit object. The category $D(\tilde{G})\text{-}\mathrm{mod}$ is graded by $\mathbb{Z}/2\mathbb{Z}$. The component of degree $0$ contains the $30$ simple modules labelled by
\begin{align*}
  (&1,1),(1,\varepsilon),(1,\chi_2),(1,\chi_3),(1,\chi_3'),(z,1),(z,\varepsilon),(z,\chi_2),(z,\chi_3),(z,\chi_3'),\\
  (&t_1,1),(t_1,-1),(t_1t_3,1),(t_1t_3,\zeta_8^2),(t_1t_3,\zeta_8^4),(t_1t_3,\zeta_8^6),(t_1t_2,1),(t_1t_2,\zeta_6^2),(t_1t_2,\zeta_6^4),\\
  (&zt_1t_2,1),(zt_1t_2,\zeta_6^2),(zt_1t_2,\zeta_6^4),(t_1t_2t_3,1),(t_1t_2t_3,\zeta_8^2),(t_1t_2t_3,\zeta_8^4),(t_1t_2t_3,\zeta_8^6),\\
  (&zt_1t_2t_3,1),(zt_1t_2t_3,\zeta_8^2),(zt_1t_2t_3,\zeta_8^4),(zt_1t_2t_3,\zeta_8^6),
\end{align*}
and the component of degree $1$ contains the $26$ others labelled by
\begin{align*}
  (&1,\psi_2),(1,\psi_2'),(1,\psi_4),(z,\psi_2),(z,\psi_2'),(z,\psi_4),(t_1,\zeta_4),(t_1,\zeta_4^3),\\
  (&t_1t_3,\zeta_8),(t_1t_3,\zeta_8^3),(t_1t_3,\zeta_8^5),(t_1t_3,\zeta_8^7),(t_1t_2,\zeta_6),(t_1t_2,\zeta_6^3),(t_1t_2,\zeta_6^5),\\
  (&zt_1t_2,\zeta_6),(zt_1t_2,\zeta_6^3),(zt_1t_2,\zeta_6^5),(t_1t_2t_3,\zeta_8),(t_1t_2t_3,\zeta_8^3),(t_1t_2t_3,\zeta_8^5),(t_1t_2t_3,\zeta_8^7),\\
  (&zt_1t_2t_3,\zeta_8),(zt_1t_2t_3,\zeta_8^3),(zt_1t_2t_3,\zeta_8^5),(zt_1t_2t_3,\zeta_8^7).
\end{align*}

We saw in proposition \ref{prop:fixed-points-ext} that the simple objects which are fixed under tensorization by $\varepsilon$ are the $V_{g,\rho}$ with $[g]=[zg]$ and $^{r}\!\rho\simeq \rho$, where $r\in \tilde{G}$ with $zg=rgr^{-1}$. Only the conjugacy classes of $t_1$ and of $t_1t_3$ satisfy the first condition.

Let us start with the study of the representations of the form $V_{t_1,\rho}$. The centralizer $Z_{\tilde{G}}(t_1)$ is isomorphic to $\mathbb{Z}/4\mathbb{Z}$ and $t_1$ generates this group. We check that $t_3t_1t_3^{-1}=zt_1$ and therefore we look at the representations $\rho$ of $Z_{\tilde{G}}(t_1)$ which satisfy ${}^{t_3}\!\rho\simeq \rho$. Since $t_3t_1t_3^{-1}=zt_1=t_1^{-1}$, only the representation labelled by $1$ and $\zeta_4^2$ satisfy this condition.

Now, we continue with the representations of the form $V_{t_1t_3,\rho}$. The centralizer $Z_{\tilde{G}}(t_1t_3)$ is isomorphic to $\mathbb{Z}/8\mathbb{Z}$ and is generated by $h:=t_1t_2t_1t_3t_2$. We check that $t_3t_1t_3t_3^{-1}=zt_1t_3$ and therefore we look at the representations $\rho$ of $Z_{\tilde{G}}(t_1t_3)$ which satisfy ${}^{t_3}\!\rho\simeq \rho$. Since $t_3ht_3^{-1}=h^{-1}$, only the representation labelled by $1$ and $\zeta_8^4$ satisfy this condition.

Finally, we obtain $4$ simple representations which are fixed under tensorization by $\varepsilon$, which agrees with Theorem \ref{thm:S1gSg1}. It only remains to compute the non-trivial orbits of isomorphism classes of simple objects under tensorization by $\varepsilon$.

In degree $0$, we have the following isomorphisms:
\begin{align*}
  &\varepsilon\otimes V_{1,1} \simeq V_{z,1}, &
  &\varepsilon\otimes V_{1,\varepsilon} \simeq V_{z,\varepsilon},&
  &\varepsilon\otimes V_{1,\chi_2} \simeq V_{z,\chi_2}, \\
  &\varepsilon\otimes V_{1,\chi_3} \simeq V_{z,\chi_3},&
  &\varepsilon\otimes V_{1,\chi_3'} \simeq V_{z,\chi_3'}, &
  &\varepsilon\otimes V_{t_1t_3,\zeta_8^2} \simeq V_{t_1t_3,\zeta_8^6},\\
  &\varepsilon\otimes V_{t_1t_2,1} \simeq V_{zt_1t_2,1},&
  &\varepsilon\otimes V_{t_1t_2,\zeta_6^2} \simeq V_{zt_1t_2,\zeta_6^2},&
  &\varepsilon\otimes V_{t_1t_2,\zeta_6^4} \simeq V_{zt_1t_2,\zeta_6^4},\\
  &\varepsilon\otimes V_{t_1t_2t_3,1} \simeq V_{zt_1t_2t_3,1},&
  &\varepsilon\otimes V_{t_1t_2t_3,\zeta_8^2} \simeq V_{zt_1t_2t_3,\zeta_8^2},&
  &\varepsilon\otimes V_{t_1t_2t_3,\zeta_8^4} \simeq V_{zt_1t_2t_3,\zeta_8^4},\\
  &\varepsilon\otimes V_{t_1t_2t_3,\zeta_8^6} \simeq V_{zt_1t_2t_3,\zeta_8^6},
\end{align*}
and in degree $1$ we have the following isomorphisms:
\begin{align*}
  &\varepsilon\otimes V_{1,\psi_2} \simeq V_{z,\psi_2}, &
  &\varepsilon\otimes V_{1,\psi_2'} \simeq V_{z,\psi_2'},&
  &\varepsilon\otimes V_{1,\psi_4} \simeq V_{z,\psi_4}, \\
  &\varepsilon\otimes V_{t_1,\zeta_4} \simeq V_{t_1,\zeta_4^3},&
  &\varepsilon\otimes V_{t_1t_3,\zeta_8} \simeq V_{t_1t_3,\zeta_8^7}, &
  &\varepsilon\otimes V_{t_1t_3,\zeta_8^3} \simeq V_{t_1t_3,\zeta_8^5},\\
  &\varepsilon\otimes V_{t_1t_2,\zeta_6} \simeq V_{zt_1t_2,\zeta_6},&
  &\varepsilon\otimes V_{t_1t_2,\zeta_6^3} \simeq V_{zt_1t_2,\zeta_6^3},&
  &\varepsilon\otimes V_{t_1t_2,\zeta_6^5} \simeq V_{zt_1t_2,\zeta_6^5},\\
  &\varepsilon\otimes V_{t_1t_2t_3,\zeta_8} \simeq V_{zt_1t_2t_3,\zeta_8},&
  &\varepsilon\otimes V_{t_1t_2t_3,\zeta_8^3} \simeq V_{zt_1t_2t_3,\zeta_8^3},&
  &\varepsilon\otimes V_{t_1t_2t_3,\zeta_8^5} \simeq V_{zt_1t_2t_3,\zeta_8^5},\\
  &\varepsilon\otimes V_{t_1t_2t_3,\zeta_8^7} \simeq V_{zt_1t_2t_3,\zeta_8^7}.
\end{align*}

We do the following choice of representatives of orbits of isomorphism classes of simple objects of degree $0$: 
\[
  V_{z,\chi_3},V_{t_1t_2,1},V_{z,\chi_2},V_{1,1},V_{zt_1t_2t_3,1},V_{t_1t_2t_3,\zeta_8^6},V_{t_1t_2t_3,\zeta_8^2},V_{t_1t_3,\zeta_8^6},V_{z,\chi_3'},V_{z,\varepsilon},V_{t_1t_2,\zeta_6^4},V_{t_1t_2,\zeta_6^2},V_{zt_1t_2t_3,\zeta_8^4}%[ "(2a,X.6)", "(6a,1)", "(2a,X.3)", "(1,1)", "(8b,1)", "(8a,X.7)", "(8a,X.4)", "(4a,X.4)", "(2a,X.7)", "(2a,X.2)", "(6a,X.5)", "(6a,X.3)", "(8b,X.2)" ] dans GAP avec les labels de DrinfeldDouble(SmallGroup(48,28))
\]
and the following choice for the degree $1$:
\begin{align*}
  &V_{t_1t_3,\zeta_8},V_{t_1t_3,\zeta_8^3},V_{t_1,\zeta_4},V_{1,\psi_4},V_{1,\psi_2'},V_{1,\psi_2},V_{t_1t_2t_3,\zeta_8^5},V_{t_1t_2t_3,\zeta_8^3},V_{t_1t_2t_3,\zeta_8^7},V_{t_1t_2t_3,\zeta_8},V_{zt_1t_2,\zeta_6},V_{zt_1t_2,\zeta_6^5},\\
  &V_{zt_1t_2,\zeta_6^3}.%[ "(4a,X.3)", "(4a,X.5)", "(4b,X.3)", "(1,X.8)", "(1,X.5)", "(1,X.4)", "(8a,X.6)", "(8a,X.5)", "(8a,X.8)", "(8a,X.3)", "(3a,X.6)", "(3a,X.4)", "(3a,X.2)" ] idem
\end{align*}

With these choices and orders, we obtain
\[
  S_{0,1}=\left(
  \begin{smallmatrix}
    6&6&12&-12&-6&-6&-6&-6&-6&-6&0&0&0\\
    0&0&0&-8&8&8&0&0&0&0&8&8&-16\\
    -12&-12&0&-8&-4&-4&0&0&0&0&8&8&8\\
    6&6&12&4&2&2&6&6&6&6&8&8&8\\
    6\sqrt 2&-6\sqrt 2&0&0&-6\sqrt 2&6\sqrt 2&6\sqrt 2&6\sqrt 2&-6\sqrt 2&-6\sqrt 2&0&0&0\\
    6\sqrt 2&-6\sqrt 2&0&0&6\sqrt 2&-6\sqrt 2&-6\sqrt 2&6\sqrt 2&-6\sqrt 2&6\sqrt 2&0&0&0\\
    6\sqrt 2&-6\sqrt 2&0&0&6\sqrt 2&-6\sqrt 2&6\sqrt 2&-6\sqrt 2&6\sqrt 2&-6\sqrt 2&0&0&0\\
    0&0&0&0&0&0&12&-12&-12&12&0&0&0\\
    6&6&-12&-12&-6&-6&6&6&6&6&0&0&0\\
    -6&-6&12&-4&-2&-2&6&6&6&6&-8&-8&-8\\
    0&0&0&-8&8&8&0&0&0&0&8&-16&8\\
    0&0&0&-8&8&8&0&0&0&0&-16&8&8\\
    6\sqrt 2&-6\sqrt 2&0&0&-6\sqrt 2&6\sqrt 2&-6\sqrt 2&-6\sqrt 2&6\sqrt 2&6\sqrt 2&0&0&0\\
  \end{smallmatrix}\right)
\]
and the diagonal matrices $T_0$ and $T_1^2$ given by the action of the inverse of the twist are
\[
  T_0=\diag(1, 1, 1, 1, 1, \zeta_4, -\zeta_4, -1, 1, 1, \zeta_3, \zeta_3^2, -1),
\]
and
\[
  T_1^2=\diag(-1, -1, -1, 1, 1, 1, -\zeta_4, \zeta_4, \zeta_4, -\zeta_4, \zeta_3^2, \zeta_3, 1).
\]

The category $D(\tilde{G})\text{-}\mathrm{mod}$ is spherical, so that the invertible object $\bar{\mathbf{1}}$ is the unit object $\mathbf{1}$. Finally, we renormalize $S_{0,1}$ by a factor $24$ which is the positive square root of $\frac{\dim((D(\tilde{G})\text{-}\mathrm{mod})_0)}{2}$, and we denote by $\tilde{S}_{0,1}$ the matrix $\frac{S_{0,1}}{24}$.

\begin{theoreme}
  \label{thm:2F4}
  The matrix $\tilde{S}_{0,1}$ is equal to the Fourier matrix of the big family of unipotent characters of ${}^2\!F_4$. Moreover the matrix $T_1^2$ is the opposite of the diagonal matrix of eigenvalues of the Frobenius. More precisely, we have:
\[
   \tilde{S}_{0,1}=S,\ F_1=T_0 \quad \text{and} \quad F_2=-T_1^{-2}.
\]
\end{theoreme}

This theorem gives another proof of \cite[Theorem 6.9, (F5), (F6)']{geck-malle} which does not require the explicit computation of the structure constants or the product $(F_2SF_1^{-1}S)^2$.

\begin{remark}
  \label{rk:why-the-group}
  The appearance of the binary octahedral group is still quite mysterious, but not the one of $\mathfrak{S}_4$. Indeed, there exists a unipotent class in the reductive group $F_4(\mathbb{C})$, denoted by $F_4(a_3)$ in \cite[Section 13.3]{carter}, whose component group of the centralizer is $\mathfrak{S}_4$, and which is the special class associated to the corresponding family of unipotent characters.
\end{remark}

%%% Local Variables:
%%% mode: latex
%%% TeX-master: "article"
%%% End:

%%% Local Variables:
%%% mode: latex
%%% TeX-master: "article"
%%% End:

\bibliographystyle{smfalpha}
\bibliography{biblio}

\end{document}